\documentclass[12pt,reqno]{amsart}
\usepackage{exscale}
\usepackage{relsize}
\usepackage{amsfonts}
\usepackage{a4wide}
\usepackage{mathrsfs,amsmath,makecell}
\allowdisplaybreaks
\numberwithin{equation}{section}
\newtheorem{theorem}{Theorem}[section]
\newtheorem{proposition}[theorem]{Proposition}
\newtheorem{corollary}[theorem]{Corollary}
\newtheorem{lemma}[theorem]{Lemma}

\theoremstyle{definition}

\newtheorem{definition}[theorem]{Definition}
\newtheorem{remark}[theorem]{Remark}

\newcommand{\R}{\mathbb{R}}

\begin{document}
\title
 [Normalized solutions for a fourth-order Schr\"{o}dinger equation]
 {Normalized solutions for a fourth-order Schr\"{o}dinger equation with positive second-order dispersion coefficient} {\let\thefootnote\relax \footnotetext{This work was supported by National Natural Science Foundation of China (Grant Nos. 11901147, 11771166) and the excellent doctorial dissertation cultivation from Central China Normal University (Grant No. 2019YBZZ064).}}

\maketitle
\begin{center}
\author{Xiao Luo}
\footnote{Email addresses: luoxiaohf@163.com (X. Luo).}

\author{Tao Yang}
\footnote{Email addresses: yangtao\_pde@163.com (T. Yang).}
\end{center}

\begin{center}
\address {1 School of Mathematics, Hefei University of Technology, Hefei, 230009, P. R. China}

\address {2 School of Mathematics and Statistics, Central China Normal University, Wuhan, 430079, P. R. China}
\end{center}
\maketitle

\begin{abstract}
We are concerned with the existence and asymptotic properties of solutions to the following fourth-order Schr\"{o}dinger equation
\begin{equation}\label{1}
{\Delta}^{2}u+\mu \Delta u-{\lambda}u={|u|}^{p-2}u, ~~~~x \in \R^{N}\\
\end{equation}
under the normalized constraint
 $$\int_{{\mathbb{R}^N}} {{u}^2}=a^2,$$
where $N\!\geq\!2$, $a,\mu\!>\!0$, $2+\frac{8}{N}\!<\!p\!<\! 4^{*}\!=\!\frac{2N}{(N-4)^{+}}$ and $\lambda\in\R$ appears as a Lagrange multiplier. Since the second-order dispersion term affects the structure of the corresponding energy functional
$$
E_{\mu}(u)=\frac{1}{2}{||\Delta u||}_2^2-\frac{\mu}{2}{||\nabla u||}_2^2-\frac{1}{p}{||u||}_p^p
$$
we could find at least two normalized solutions to (\ref{1}) if $2\!+\!\frac{8}{N}\!<\! p\!<\!{ 4^{*} }$ and $\mu^{p\gamma_p-2}a^{p-2}\!<\!C$ for some explicit constant $C\!=\!C(N,p)\!>\!0$ and $\gamma_p\!=\!\frac{N(p\!-\!2)}{4p}$.
Furthermore, we give some asymptotic properties of the normalized solutions to (\ref{1}) as $\mu\to0^+$ and $a\to0^+$, respectively. In conclusion, we mainly extend the results in \cite{DBon,dbJB}, which deal with (\ref{1}), from $\mu\leq0$ to the case of $\mu>0$, and also extend the results in \cite{TJLu,Nbal}, which deal with (\ref{1}), from $L^2$-subcritical and $L^2$-critical setting to $L^2$-supercritical setting.



\end{abstract}

\maketitle

\section{Introduction and Main Result}

\setcounter{equation}{0}

This paper concerns the existence of solutions $({\lambda},u)\in  {\mathbb{R}}\times H^2({\mathbb{R}^N})$ to the following fourth-order Schr\"{o}dinger equation
\begin{equation}\label{eq1.1}
{\Delta}^{2}u+\mu \Delta u-{\lambda}u={|u|}^{p-2}u, ~~~~x \in \R^{N}
\end{equation}
under the constraint
\begin{equation}\label{eq1.2}
  \int_{{\mathbb{R}^N}} {{u}^2}=a^2,
\end{equation}
where $N\geq 2$, $a>0$, $\mu>0$, $2+\frac{8}{N}< p<{ 4^{*} }$. Here
$${ 4^{*} }=+\infty~~\mbox{if}~~N\leq4,
~~\mbox{and}~~{ 4^{*} }=\frac{2N}{N-4}~~\mbox{if}~~N\geq5.$$
We call $u$ a normalized solution to (\ref{eq1.1}), since (\ref{eq1.2}) imposes a normalization on its $L^2$-mass. Normalized solutions to (\ref{eq1.1}) can be obtained by searching critical points of
$$
E_{\mu}(u)=\frac{1}{2}{||\Delta u||}_2^2-\frac{\mu}{2}{||\nabla u||}_2^2-\frac{1}{p}{||u||}_p^p
$$
on the constraint
\begin{equation} \label{eq1.06}
S_{a}: =\Big\{ u \in H^2({\mathbb{R}^N}): {||u||}_2^2=\int_{{\mathbb{R}^N}} {u}^2=a^2  \Big\}
\end{equation}
with $\lambda$ appearing as Lagrange multipliers. This fact implies that $\lambda$ cannot be determined a priori, but is part of the unknown.

In \cite{DCas}, D. Bonheure et al. studied the fourth order nonlinear Helmholtz equation
\begin{equation} \label{eQA1.06}
\Delta^{2} u+\mu\Delta u-\lambda u=\Gamma|u|^{p-2} u ~~~~\mbox{in}~~~~ \mathbb{R}^{N}
\end{equation}
for $2< p<{ 4^{*} }$ and positive, bounded and $\mathbb{Z}^N$-periodic functions $\Gamma$ in the following three cases: (1) $\lambda>0$, $\mu\in \mathbb{R} $; (2) $\lambda<0$, $\mu>2\sqrt{-\lambda}$; (3) $\lambda=0$, $\mu>0$. Using the dual method introduced by G. Evequoz et al. in \cite{QuOz}, they found solutions to \eqref{eQA1.06} and established some of their qualitative properties. Since $\lambda$ is prescribed in \eqref{eQA1.06}, problem (\ref{eq1.1})-(\ref{eq1.2}) is different from that of \cite{DCas}.

Problem (\ref{eq1.1})-(\ref{eq1.2}) arises from seeking standing waves for the time-dependent fourth-order Schr\"{o}dinger equation
\begin{equation}\label{eq1.3}
i\partial_t{\psi}-{\Delta}^{2}{\psi}-\mu \Delta \psi+ {|\psi|}^{p-2}{\psi}=0, ~~~~\psi(0,x)=\psi_0(x),~~~~(t,x)\in {\mathbb{R}}\times {\mathbb{R}^N}.
\end{equation}
A standing wave of (\ref{eq1.3}) is a solution having the form ${\psi}(t,x)=e^{-i {\lambda} t}u(x)$ for some ${\lambda}\in \mathbb{R}$ and $u$ solving (\ref{eq1.1}). So (\ref{eq1.1}) is the stationary equation of (\ref{eq1.3}).

Dating back to \cite{VIKA}, V. I. Karpman introduced the following equation
\begin{equation}\label{eq1.4}
i\partial_t{\psi}-\gamma {\Delta}^{2}{\psi}-\mu \Delta \psi+ {|\psi|}^{p-2}{\psi}=0, ~~~~\psi(0,x)=\psi_0(x),~~~~(t,x)\in {\mathbb{R}}\times {\mathbb{R}^N}
\end{equation}
with the corresponding stationary equation
\begin{equation}\label{eq1.401}
\gamma {\Delta}^{2}u+\mu \Delta u-{\lambda}u={|u|}^{p-2}u, ~~~~x \in \R^{N},
\end{equation}
where $\gamma>0$ and $\mu=-1$. It is well known that two quantities are conserved in time along trajectories of (\ref{eq1.4}): the energy
\begin{align*}
E_{\mu,\gamma}(u)=\frac{\gamma}{2}{||\Delta u||}_2^2-\frac{\mu}{2}{||\nabla u||}_2^2-\frac{1}{p}{||u||}_p^p
\end{align*}
and the mass $\int_{{\mathbb{R}^N}} {{{|u}}|^2}$. In particular, taking $\gamma\!=\!0$ and $\mu\!=\!-1$ in (\ref{eq1.4}), we recover the nonlinear Schr\"{o}dinger equation
\begin{equation}\label{eq1.05}
i\partial_t{\psi}+\Delta \psi+ {|\psi|}^{p-2}{\psi}=0, ~~~~\psi(0,x)=\psi_0(x),~~~~(t,x)\in {\mathbb{R}}\times {\mathbb{R}^N}.
\end{equation}
In nonlinear optics, equation (\ref{eq1.05}) (with $p=4$, $N=2$) is derived from the scalar nonlinear Helmhotz equation through the paraxial approximation. Physicists use (\ref{eq1.05}) to describe the canonical model for propagation of intense laser beams in a bulk medium with Kerr nonlinearity, see \cite{DBon, GfbI}. It is well-known that when $2<p<\frac{4}{N}+2$, solutions to (\ref{eq1.05}) exist globally in time and that they are stable whereas when if $\frac{4}{N}+2 \leq p<\frac{2N}{N-2}$, they can become singular in finite time and they are
unstable, see \cite{NaVe,CSul,TcAz,MiWe}. In fact, if $\frac{4}{N}+2< p<\frac{2N}{N-2}$, the ground states (intended here as the least energy solutions associated with the free functional, where the $\lambda$ is fixed) are unstable, but the problem of stability/instability is still open for the exited states (namely for the solutions which are not ground states).

In order to regularize and stabilize the solutions to (\ref{eq1.05}), V. I. Karpman \cite{VIKA} added a fourth-order dispersion term to (\ref{eq1.05}) and studied (\ref{eq1.4}) with $\gamma>0$ and $\mu=-1$. It is known that the Cauchy problem (\ref{eq1.4}) is locally well-posed in $H^2(\R^N)$ provided $2<p<4^*$(See \cite{Bpau, MBHK, Cken}). One can refer to \cite{Bpau, BPau, BPAu, BpSx, MBHK, Qguo, CmGl} for well-posedness and scattering, and \cite{TbOE} concerning finite-time blow up. The one-dimensional stationary mixed dispersion nonlinear Schr\"{o}dinger equation also arises in the theory of water waves(See \cite{BBUf, BbUf}). From \cite{VIKA, ViKv, EWlK, ewLa}, we concluded that: when $2<p<\frac{4}{N}+2$ for $\gamma>0$, or $\frac{4}{N}+2\leq p<\frac{8}{N}+2$ for $\gamma>0$ small enough, standing waves of (\ref{eq1.4}) are stable, and when $p>\frac{8}{N}+2$ for $\gamma>0$, they become unstable. In \cite{GfbI}, G. Fibich et al. proved that when $2<p<\frac{8}{N}+2$, all solutions to (\ref{eq1.4}) exist globally in time. On the other hand, they mentioned that existence of blowing up solutions to (\ref{eq1.4}) for $p\geq\frac{8}{N}+2$ is a difficult open problem, which has now been recently partially solved by  T. Boulenger et al. in \cite{TbOE}.  When $p>\frac{8}{N}+2$, T. Boulenger et al. proved a general result on finite-time blow up for radial data in any dimension $N\geq2$. Moreover, they derived a universal upper bound for the blow up rate for suitable $\frac{8}{N}+2<p<4^*$. For $p=\frac{8}{N}+2$, they proved a general blow up result in finite or infinite time for radial data. Later on, D. Bonheure et al. in \cite{DboN} proved that radial least energy solutions are unstable and thus complemented the results of \cite{TbOE}.

Recently, more and more attention are paid to the existence of normalized solutions to (\ref{eq1.401}), especially for ground states, see \cite{DBon, TJLu,Nbal,dbJB}. Following Definition 1.1 in \cite{iniL}, we say that $u$ is a \textbf{ground state} of (\ref{eq1.401}) on $S_a$ if it is a solution to (\ref{eq1.401}) having minimal energy among all solutions which belong to $S_a$, that is
$$
d\left.E_{\mu,\gamma}\right|_{S_{a}}(u)=0 \quad \text { and } \quad E_{\mu,\gamma}(u)=\inf \left\{E_{\mu,\gamma}(w): d\left.E_{\mu,\gamma}\right|_{S_{a}}(w)=0 \text { and }  w\in S_{a}\right\}.
$$
Furthermore, as in \cite{dbJB}, we define $u$ a \textbf{radial ground state} of (\ref{eq1.401}) on $S_{a,r}\!=\!S_a \cap H_{rad}^2$ provided
$$
d\left.E_{\mu,\gamma}\right|_{S_{a}}(u)=0 \quad \text { and } \quad E_{\mu,\gamma}(u)=\inf \left\{E_{\mu,\gamma}(w): d\left.E_{\mu,\gamma}\right|_{S_{a,r}}(w)=0 \text { and }  w\in S_{a,r}\right\}.
$$

In \cite{DBon}, D. Bonheure et al. studied ground state of (\ref{eq1.401}) on $S_a$ with $\gamma>0$, $\mu\leq0$ and $2<p<\frac{8}{N}+2$ by utilising the constrained minimization method since $\inf _{u \in S_{a}} E_{\mu,\gamma}(u)\!>\!-\infty$. They focused on the existence results, qualitative properties, exponential decay and orbital stability. Fruitful achievements have been made in their work.

In \cite{TJLu}, T. Luo et al. considered (\ref{eq1.401}) with $\gamma=1$, $\mu \in \R$ and $2<p\leq\frac{8}{N}+2$. They studied the minimization problem
\begin{equation}\label{qno.8}
m(a,\mu,\gamma):=\inf _{u \in S_{a}} E_{\mu,\gamma}(u),
\end{equation}
by using the profile decomposition of bounded sequences in $H^2$ established in \cite{ShZj}. T. Luo et al. showed that $m(a,\mu,\gamma)$ is achieved in five cases: (1)~~$a=1$, $2<p<2+\frac{4}{N}$ and $\mu \in(-\infty,0)$; (2)~~$a=1$, $2+\frac{4}{N} \leq p<2+\frac{8}{N}$ and $\mu \in\left(-\mu_{0}, 0\right)$ for some $\mu_{0}>0$; (3)~~$a=1$, $2<p<2+\frac{8}{N}$ and $\mu=0$; (4)~~$a=1$, $2<p<2+\frac{8}{N}$ and $\mu \in\left(0, \hat{\mu}_{0} \right]$ for some $\hat{\mu}_{0}>0$; (5)~~$0\!<\!a\!<\!a_N$, $p\!=\!2+\frac{8}{N}$ and $\mu \in\left(0, \tilde{{\mu}}_{0} \right]$ for some $a_N, \tilde{{\mu}}_{0}>0$.

In \cite{Nbal}, N. Boussaid et al. studied \eqref{qno.8} with $\gamma\!>\!0$, $\mu\!>\!0$ and $2\!<\!p\!\leq \!\frac{8}{N}\!+\!2$. Under this situation, they improved the results of \cite{TJLu} by relaxing the extra restriction on $a$ and $\mu$. In \cite{TJLu}, there is an explicit lower bound on $a\!>\!0$ and an upper bound on $\mu\!>\!0$. N. Boussaid et al. solved \eqref{qno.8} for all $a>0$ and $\mu>0$ when $2\!<\!p\!\leq \!\frac{8}{N}\!+\!2$. The problem is especially difficult for $a>0$ small. The key point is to rule out the vanishing of the minimizing sequences since it also exclude the possibility of dichotomy. They give a necessary and sufficient condition to avoid the vanishing, i.e. $m(a,\mu,\gamma)<-\frac{a^2{\mu}^2}{8\gamma}$.

D. Bonheure et al. in \cite{dbJB} considered (\ref{eq1.401}) with $\gamma>0$, $\mu=-1$ and $\frac{8}{N}+2\leq p <4^*$. In this case, it is no more possible to obtain a critical point of $E_{\mu,\gamma}$ restricted to $S_a$ as a global minimizer. Fortunately, D. Bonheure et al. in \cite{dbJB} discovered that $E_{\mu,\gamma}$ restricted to $S_a$ possesses a natural constraint, namely a set, that contains all the critical points of $E_{\mu,\gamma}$ restricted to $S_a$. Therefore, they concerned the existence of minimizers associated to $\Gamma(a) :=\inf _{u \in \mathcal{M}(a)} E_{\mu,\gamma}(u)$
where
$$\mathcal{M}(a) :=\Big \{ u \in H^2({\mathbb{R}^N}): {||u||}_2^2=a^2,~~~~ P_{\mu,\gamma}(u)=0 \Big \}$$
and
$$P_{\mu,\gamma}(u):=2\gamma{||\Delta u||}_2^2-\mu{||\nabla u||}_2^2-2{\gamma_p}{||u||}_p^p=0$$
is the related Pohozaev identity with $\gamma_p=\frac{N(p-2)}{4p}$. They proved that $\Gamma(a) :=\inf _{u \in \mathcal{M}(a)} E_{\mu,\gamma}(u)$ is attained provided $a_0<a<a_{N,p}$ for some $a_0 $ and $a_{N,p}$ satisfing $a_{N,p}>a_0>0$ and the minimizers are ground states of (\ref{eq1.401}) on $S_a$. In addition, they proved the existence of infinitely many radial normalized solutions to (\ref{eq1.401}).

As in \cite{DBon,TJLu,Nbal,dbJB}, by the $L^2$-norm preserving dilations $u_{t}(x)=t^{\frac{N}{2}}u(tx)$ with $t\!>\!0$, it is easy to know that $\bar{p}:=\!\frac{8}{N}\!+\!2$ is the $L^2$-critical exponent of (\ref{eq1.401}) since $\inf _{u \in S_{a}} E_{\mu,\gamma}(u)\!=\!-\infty$ if $\bar{p}<p\leq4^*$ and $\inf _{u \in S_{a}} E_{\mu,\gamma}(u)\!>\!-\infty$ if $2\!<\!p<\!\bar{p}$. To our best knowledge, the existing results on normalized solutions to (\ref{eq1.401}) can be summarized in the following Table (with additional conditions on the parameters):
$$
\begin{array}   {|c|c|c|c|c|}
\hline \text { $\gamma$ } & {\text { $\mu$ }} & {\text {$p$}} & {\text { Number and type of solutions }} & {\text { References }} \\
\hline \text { $\gamma>0$ } & {\text { $\mu\leq 0$ }} & {\text {$2<p<\bar{p}$}} & {\text { A ground state  }} & {\text { \cite{DBon} }} \\
\hline \text { $\gamma=1$ } & {\text { $\mu\!\leq\!0, 0\!<\!\mu\!<\!\mu_0$ }} & {\text {$2<p\leq \bar{p}$}} & {\text { A ground state  }} & {\text { \cite{TJLu} }} \\
\hline \text { $\gamma>0$ } & {\text { $\mu=-1$ }} & {\text {$\bar{p}\leq p<4^*$}} & \makecell[c]{\text{A ground state};\\ \text{infinitely radial solutions}} & \text{\cite{dbJB}} \\
\hline \text { $\gamma>0$ } & {\text { $\mu>0$ }} & {\text {$2<p\leq\bar{p}$}} & {\text{ A ground state }} &  {\textbf{ \cite{Nbal} }}\\
\hline \text { $\gamma>0$ } & {\text { $\mu>0$ }} & {\text {$\bar{p}<p\leq4^*$}} & {\textbf{ Unknown }} &  {\textbf{ Unknown }}\\
\hline
\end{array}
$$

In this paper, we consider the existence and asymptotic properties of normalized solutions to (\ref{eq1.401}) with $\gamma>0$, $\mu>0$ and $\bar{p}<p<4^*$.
For the sake of convenience, we take $\gamma\!=\!1$ since the coefficient $\gamma\!>\!0$ in (\ref{eq1.401}) can be scaled out. In fact, setting $v(x)\!=\!u(\gamma^{\frac{1}{4}} x)$ with $u$ solving (\ref{eq1.401}), then $v$ solves (\ref{eq1.1}) with $\mu$ replaced by $\frac{\mu}{\sqrt{\gamma}}$.\\

To state our main results, let us recall the Gagliardo-Nirenberg inequality, for $2<p<4^{*}$
\begin{equation} \label{2.9}
  {||u||}_p^{p} \leq C_{N,p}^p {||u||}_2^{p(1-\gamma_p)} {||\Delta u||}_2^{p\gamma_p},\quad  \forall u\in H^{2}(\R^N)
\end{equation}
where $C_{N,p}$ is some positive constant and $\gamma_p=\frac{N(p-2)}{4p}$. One can refer to \cite{TbOE,dbJB,TJLu,Nbal} for more information about \eqref{2.9}. For $\bar{p}<p<4^{*}$, we introduce three positive constants:
\begin{equation*} 
\tilde{C}(N,p):=\!\frac{p}{2(p\gamma_p\!-\!1)C_{N,p}^p}{\Big( \frac{p\gamma_p\!-\!2}{p\gamma_p\!-\!1} \Big)}^{p\gamma_p\!-\!2}
\end{equation*}
and
\begin{equation*}
{C}^{*}(N,p):=\!\frac{2^{p\gamma_p\!-\!2}p}
{2(p\gamma_p\!-\!1)C_{N,p}^p}\Big(\frac{1\!-\!\gamma_p}{\gamma_p} \Big) ^{\frac{p\gamma_p\!-\!2}{2}},~~~~~~~~{C}_{*}(N,p):=\frac{p}
{2(p\gamma_p\!-\!1)C_{N,p}^p}\Big[ \frac{2(p\!-\!2)}{p\gamma_p\!-\!1} \Big] ^{\frac{p\gamma_p\!-\!2}{2}}.
\end{equation*}

Our main results are as follows.

\begin{theorem}\label{th1.1}
Let $N\!\geq\!5$, $\overline{p}\!<\!p\!<\!\min\{4,4^*\}$ and $a,\mu>0$ such that
${\mu}^{p\gamma_p-2}{a}^{p-2}<\min \{\tilde{C}(N,p),{C}^{*}(N,p), {C}_{*}(N,p)\}$. Then \\
$\textbf{(1)}$ $E_{\mu}|_{S_{a}}$ has a critical point $\tilde{u}_{\mu}$ at level $m_r(a, \mu)<-\frac{a^2{\mu}^2}{8}$, which is an interior local minimizer of $E_{\mu}$ on the set
$$
A_{R_0}^r :=\left\{u \in S_{a}\cap H_{{rad}}^{2} : {||\Delta u||}_2<R_0\right\}
$$
for a suitable $R_0=R_0(a,\mu)>0$. Moreover, $\tilde{u}_{\mu}$ is a radial ground state of (\ref{eq1.1}) on $S_{a,r}$, and any other radial ground state is a local minimizer of $E_{\mu}$ on $A^r_{R_0}$. \\
$\textbf{(2)}$ $E_{\mu}|_{S_{a}}$ has a second critical point of mountain pass type $\hat{u}_{\mu}$ at a positive level $\sigma(a, \mu)\!>\!0$.\\
$\textbf{(3)}$ Both $\tilde{u}_{\mu}$ and $\hat{u}_{\mu}$ are real-valued radial solutions to (\ref{eq1.1}) for suitable $\tilde{\lambda}, \hat{\lambda}<-\frac{{\mu}^2}{4}$;
Suppose in addition that $N\!<\!8$ and $p\!<\!\min \{\frac{2(N-2)}{N-4},4\}$, then $\tilde{u}_{\mu}$ and $\hat{u}_{\mu}$ are sign-changing.\\
$\textbf{(4)}$ $m_r(a, \mu)\!\to\! 0^-$, $\tilde{\lambda}\to 0^-$ and any radial ground state $\tilde{u}_{\mu} \in S_{a,r}$ for $\left.E_{\mu}\right|_{S_{a}}$ satisfies $||\Delta \tilde{u}_{\mu}||_{2} \rightarrow 0$ as $\mu \!\to\! 0^{+}$.\\
$\textbf{(5)}$ $\sigma(a, \mu)\!\to\!\sigma(a, 0)$ and $\hat{u}_{\mu}\!\to\!\hat{u}$ in $H^2$ as $\mu \!\to\! 0^{+}$, where $\sigma(a, 0)\!=\!E_{0}(\hat{u})$ with $\hat{u}$ being a radial ground state to (\ref{eq1.1}) obtained for $\mu=0$.
\end{theorem}

\begin{theorem}\label{th1.3}
Let $N\!\geq\!5$, $\overline{p}\!<\!p\!<\!\min\{4,4^*\}$, $a_k \!\to\! 0^+$ as $k\!\to\!+\infty$ and $u_k\!\in\! A_{R_{0}}^{r}\!=\!\left\{u \!\in\! S_{a_k} \cap H_{{rad}}^{2} : {||\Delta u||}_2\!<\!R_0(a_k,\mu) \right\}$ be a minimizer of $m_r(a_k,\mu)$ for each $k\!\in\! \mathbb{N}$. Then  \\
$\textbf{(1)}$ There exists $\{\varepsilon_k\}\!\subset\!{\R}^+$ with $\varepsilon_k \!\to\! 0$ such that $-\frac{{\mu}^2a_k^2}{8}(1\!+\!\varepsilon_k) \!\leq\! m_r(a_k, \mu)\!\leq\!-\frac{{\mu}^2a_k^2}{8}$, $\forall k\!\in\! \mathbb{N}.$ \\
$\textbf{(2)}$ The corresponding Lagrange multiplier ${\tilde{\lambda}}_k$ satisfies ${\tilde{\lambda}}_k\to (-\frac{{\mu}^2}{4})^-$ as $k\to+\infty$. \\
$\textbf{(3)}$ $\frac{||\Delta u_k||_{2}^2}{||u_k||_{2}^2} \rightarrow \frac{{\mu}^2}{4}$, $\frac{||\nabla u_k||_{2}^2}{||u_k||_{2}^2} \rightarrow \frac{\mu}{2}$ as $k\to+\infty$.\\
$\textbf{(4)}$ Let $v_k\!=\!\frac{u_k}{||u_k||_{2}}$, it results that $v_k\!\to\! 0$ in $L^q({\R}^N)$ for any $q\!\in\!(2,4^*)$. Furthermore, $\int_{\mathbb{R}^{N}}\left(|\xi|^{2}\!-\!\frac{\mu}{2 }\right)^{2}\left|(\mathcal{F}{v}_{k})(\xi)\right|^{2} d \xi \!\rightarrow \!0$ as $k \!\to\! +\infty$, where $\mathcal{F}{v}_{k}$ is the Fourier transform of ${v}_{k}$. \\
\end{theorem}

\begin{remark}\label{re1.3}
Theorem \ref{th1.1} indicates that there exists at least two normalized solutions to (\ref{eq1.1}) in the $L^2$-supercritical setting, one radial ground state and one radial excited state (whose energy is strictly larger than that of the radial ground state).  Moreover, the radial ground state to (\ref{eq1.1}) vanishes and the radial excited state converges to a radial ground state of the related limiting equation $\Delta^{2} u\!-\! {\hat{\lambda}}u\!=\!|u|^{p-2} u$ as $\mu\!\to\! 0^{+}$. These facts show that the sign of the coefficient of the second-order dispersion term has crucial effect on the structure of the energy functional $E_{\mu}$ and makes the solution set to (\ref{eq1.1}) much richer. Theorem \ref{th1.3} describes the asymptotic properties of the solutions (local minimizers) obtained in Theorem \ref{th1.1} as its mass vanishes. Theorem \ref{th1.3}-(1) implies that it is very difficult to prove
$m(a,\mu)<-\frac{a^2{\mu}^2}{8}$ for $a>0$ small, which is vital in the proof of Theorem \ref{th1.1}. From Theorem \ref{th1.3}-(4) we see that the $L^2$-norm of $\{v_k\}\subset S_1$ concentrate near the sphere of radius $\sqrt{\frac{\mu}{2 }}$ centered at the origin when $k\to+\infty$. This was one of the keys to find the test functions that allow to
show that, under the assumption $\bar{p}<p<4$, the strict inequality $m(a,\mu)<-\frac{a^2{\mu}^2}{8}$ holds for $a,\mu>0$ such that
${\mu}^{p\gamma_p-2}{a}^{p-2}<\tilde{C}(N,p)$. We also point out that the condition $m(a,\mu)<-\frac{a^2{\mu}^2}{8}$ was already observed by the work of N. Boussaid et al. in \cite{Nbal}.

Theorems \ref{th1.1}-\ref{th1.3} mainly extend the results in \cite{DBon,dbJB}, which deal with (\ref{eq1.1}), from $\mu\leq0$ to the case of $\mu>0$, and also extend the results in \cite{TJLu,Nbal}, which deal with (\ref{eq1.1}), from $L^2$-subcritical and $L^2$-critical setting $2<p\leq\bar{p}$ to $L^2$-supercritical setting $\bar{p}<p<4^*$.
\end{remark}

\begin{remark}
For $N\!\geq\!8$ and $p$ satisfying $\overline{p}\!<\!p\!<\!\min\{4,4^*\}$, we see that $p$ covers the interval $(\overline{p}, 4^*)$ since $4^*\leq4$. In Theorem \ref{th1.1}, the condition ${\mu}^{p\gamma_p-2}{a}^{p-2}\!<\!\tilde{C}(N,p)$ makes sure that $E_{\mu}$ presents a convex-concave geometry. Therefore, it is possible to expect the existence of a local minimizer and a mountain pass critical point for $E_{\mu}|_{S_{a}}$. The extra conditions ${\mu}^{p\gamma_p-2}{a}^{p-2}\!<\!\min \{{C}^{*}(N,p), {C}_{*}(N,p)\}$ are useful in searching the mountain pass critical point for $E_{\mu}|_{S_{a}}$, but they are not necessary in obtaining the local minimizer.
\end{remark}

Now we underline some of the difficulties that arise in the proof of Theorem \ref{th1.1}. Since $\inf _{u \in S_{a}} E_{\mu}(u)=-\infty$  for $\bar{p}<p<4^*$, the constrained minimization method used in \cite{DBon,TJLu,Nbal} does not work any more. Naturally, we would hope to overcome this difficulty by using the Pohozaev constraint approach adopted in \cite{iniL,dbJB,NSoa,nsoa}.

However, the compactness of a Palais-Smale sequence is a highly nontrivial issue for $\mu\!>\!0$ even in the radial space $H_{rad}^2(\R^N)$. To be precise, let $\left\{u_{n}\right\} \subset S_{a,r}\!=\!S_a \cap H_{rad}^2$ be a Palais-Smale sequence for $E_{\mu}|_{S_a}$ at level $c \neq 0$ with
$$ P_{\mu}\left(u_{n}\right):=2{||\Delta u_{n}||}_2^2-\mu{||\nabla u_{n}||}_2^2-2{\gamma_p}{||u_{n}||}_p^p=o_n(1)$$
and $u_{n} \rightharpoonup u$ in $H^2(\R^N)$. \textbf{Firstly, we shall claim $u\not \equiv0$}. Since $u \equiv0$ leads to
$$c=\mathop {\lim }\limits_{n  \to \infty} E_{\mu}\left(u_{n}\right)=\mathop {\lim }\limits_{n  \to \infty} [E_{\mu}\left(u_{n}\right)-\frac{1}{4}P_{\mu}\left(u_{n}\right)]=-\frac{\mu}{4} \mathop {\lim }\limits_{n  \to \infty}||\nabla u_{n}||_{2}^{2}+\frac{p \gamma_{p}-2} {2p}||u||_{p}^{p}\leq0,$$
it must be $u\not \equiv0$ if $c\!>\!0$. In the case of $c\!<\!0$, we need the exact upper bound $c\!<\!-\frac{a^2\mu^2}{8}$ in proving $u\not \equiv0$, see Lemma \ref{lem4.2} for details. As can be seen from Theorem \ref{th1.3}, this is especially difficult for $a\!>\!0$ small since  $$m_r(a_k,\mu)\to-\frac{a_k^2\mu^2}{8}
~~~~\mbox{with}~~~~a_k\to0^+~~~~\mbox{as}~~~~k\to+\infty.$$
In \cite{Nbal}, N. Boussaid et al. observed that $m(a,\mu,\gamma)<-\frac{a^2{\mu}^2}{8\gamma}$ is a necessary and sufficient condition to avoid the vanishing of the minimizing sequences for $m(a,\mu,\gamma)=\inf _{u \in S_{a}} E_{\mu,\gamma}(u)$ provided $2\!<\!p\!<\!\bar{p}$. Our case is more delicate since we consider $\bar{p}\!<\!p\!<\!4^*$. In this case, $m_r(a, \mu)$ is characterized by a local minimizer value rather than a global one. We use truncation skills to prove $m_r(a, \mu)\!<\!-\frac{a^2\mu^2}{8}$ and we shall always keep the testing functions staying in the admissible set, which is the key ingredient of Section 3.

Having proved $u\not \equiv0$ and using the compact embedding $H_{rad}^{2}\left(\mathbb{R}^{N}\right) \hookrightarrow L^{r}\left(\mathbb{R}^{N}\right)$ for $r \in\left(2,4^{*}\right)$, we derive that the corresponding Lagrange multipliers $\lambda_n \to\lambda<0$ and
\begin{equation} \label{eq1.08}
||\Delta( u_{n}-u)||_2^{2}-\mu||\nabla( u_{n}-u)||_2^{2}-\lambda ||u_{n}-u||_2^{2}=o_n(1).
\end{equation}
\textbf{But we can not obtain $\nabla u_{n} \rightarrow \nabla u$ in $L^2(\R^N)$ from $u_{n} \rightharpoonup u$ in $H^2(\R^N)$}. If $\mu\leq 0$, as in \cite{dbJB}, (\ref{eq1.08}) is sufficient to deduce that $u_{n} \rightarrow u$ in $H^2(\R^N)$, but we are in the setting $\mu>0$.

Assume that $||\Delta( u_{n}-u)||_2^{2}\geq\delta$ and $||u_{n}-u||_2^{2}\geq\delta$ for some $\delta>0$, otherwise compactness holds. From (\ref{eq1.08}), we have
\begin{align*}
||\Delta( u_{n}\!-\!u)||_2^{2}\!-\!\lambda ||u_{n}\!-\!u||_2^{2}\!=\!\mu||\nabla( u_{n}\!-\!u)||_2^{2}\!+\!o_n(1)\!\leq\! \mu||\Delta( u_{n}\!-\!u)||_2|| u_{n}\!-\!u||_2\!+\!o_n(1).
\end{align*}
It results that
\begin{align*}
2\sqrt{-\lambda} \leq \frac{||\Delta( u_{n}-u)||_2}{||u_{n}-u||_2}-\lambda \frac{||u_{n}-u||_2}{||\Delta( u_{n}-u)||_2} \leq \mu+o_n(1)  \Longrightarrow -\frac{{\mu}^2}{4} \leq \lambda <0.
\end{align*}
To get compactness, a possible way is to prove $\lambda<-\frac{{\mu}^2}{4}$.
We notice that the lower bound of $||\Delta u_{n}||_2$ determines the upper bound of $\lambda$, and therefore the sign of the energy level $c$ is very important in the analysis. An uniformly lower bound of $||\Delta u_{n}||_2$ follows easily from $c>0$. However, if $c<0$, we can only obtain an uniformly upper bound of $||\Delta u_{n}||_2$. The reason lies in that the upper bound of $c$ is too large. Thanks to Lemma \ref{lema3.4.1}, we have the exact upper bound $c<-\frac{a^2\mu^2}{8}$, which is sufficient to give an uniformly positive lower bound of $||\Delta u_{n}||_2$. See Lemmas \ref{lem4.1}-\ref{lem4.2} for details. \\

The paper is organized as follows, in Section 2, we give some preliminary results. In Section 3, we give the exact upper bounds of $m(a,\mu)$ and $m_r(a,\mu)$. In Section 4, we give the compactness analysis of Palais-Smale sequences. In Section 5, we prove Theorems \ref{th1.1}-\ref{th1.3}. \\

\noindent \textbf{Notations:}~~~~$D^{2,2}(\R^N)$ is the completion of $C_{c}^{\infty}(\R^N)$ under the norm $\|u\|_{{D}^{2,2}}\!=\!||\Delta u||_{2}$. $L^{p}(\R^N)~(1\!<\!p\!\leq\!\infty)$ is the Lebesgue space with norm $||u||_{p}\!=\!\Big(\int_{\R^N} {{|u(x)|}^p dx}\Big)^{\frac{1}{p}}$. Denote
$H^{2}(\R^N)\!=\!\left\{u \!\in\! L^{2}(\R^N): \nabla u, \Delta u \!\in\! L^{2}(\R^N)\right\} $ and $ H_{{rad}}^{2}(\R^N)\!=\!\left\{u \!\in\! H^{2}(\R^N) : u(x)\!=\!u(|x|)\right\}$
with the equivalent norm $\left(||u||_{2}^{2}+||\Delta u||_{2}^{2}\right)^{\frac{1}{2}}$. We use $``\rightarrow"$ and $``\rightharpoonup"$ to denote the strong and weak convergence in the related function spaces respectively. $C$ and $C_{i}$ will denote positive constants. $\langle\cdot,\cdot\rangle$ denote the dual pair for any Banach space and its dual space. $\mathbb{N}=\{1,2,\cdots\}$ is the set of natural numbers. $\mathbb{R}$ and $\mathbb{C}$ denote the sets of real and complex numbers respectively. $\overline{\Omega}$ denotes the closure of $\Omega$. $o_{n}(1)$ and $O_{n}(1)$ mean that $|o_{n}(1)|\to 0$ as $n\to+\infty$ and $|O_{n}(1)|\leq C$ as $n\to+\infty$, respectively. The Fourier transform of $u\in L^{2}(\R^N)$ is denoted by $\mathcal{F}u$.


\section{Preliminaries}

\setcounter{equation}{0}
In this section, we give some preliminary results.

We shall often use the interpolation inequality(See (2.3) in \cite{Nbal})
\begin{equation}   \label{eqa2.3.1}
{||\nabla u||}_2^2\leq{||u||}_2{||\Delta u||}_2,~~~~\forall u\in H^{2}(\R^N).
\end{equation}
For any $\mu\geq 0$, we deduce from \eqref{2.9} and \eqref{eqa2.3.1} that
\begin{equation}   \label{eq2.3}
 E_{\mu}(u)\geq \frac{1}{2}{||\Delta u||}_2^2-\frac{\mu a}{2}{||\Delta u||}_2-\frac{C_{N,p}^p}{p}a^{p(1-\gamma_p)}{||\Delta u||}_2^{p\gamma_p}, ~~~~\forall u\in S_a,
\end{equation}
which indicates that $\inf _{u \in S_{a}} E_{\mu}(u)>-\infty$  for $2<p< \bar{p}$, see \cite{DBon,TJLu,Nbal}. However, by dilations $u_{t}(x)=t^{\frac{N}{2}}u(tx)$ with $t\!>\!0$, we deduce that $\inf _{u \in S_{a}} E_{\mu}(u)=-\infty$ for $\bar{p}<p\leq4^*$. The constrained minimization method used in \cite{DBon,TJLu,Nbal} does not work any more. Naturally, we would hope to overcome this difficulty by using the Pohozaev constraint approach adopted in \cite{dbJB,NSoa,nsoa}. To this end, we introduce the Pohozaev set:
\begin{equation}  \label{eq2.4}
\mathcal{P}_{a, \mu}=\left\{u \in S_{a} : P_{\mu}(u)=0\right\}
\end{equation}
where
\begin{equation}  \label{eq2.5}
P_{\mu}(u) :=2{||\Delta u||}_2^2-\mu{||\nabla u||}_2^2-2{\gamma_p}{||u||}_p^p~~~~\mbox{for}~~~~\gamma_p=\frac{N(p-2)}{4p}.
\end{equation}
Then any critical point of $E_{\mu}|_{S_a}$ stays in $\mathcal{P}_{a, \mu}$ as a consequence of the Pohozaev identity:

\begin{proposition}(Lemma 2.1 in \cite{DboN})\label{pro3.01}
Let $\mu\!\in\!\R$, $\lambda\!<\!0$ and $2\!<\!p\!<\!4^{*}$. If $v \!\in\! H^{2}\left(\mathbb{R}^{N}\right)$ is a weak solution of 
\begin{equation} \label{eq3.01}
 \Delta^{2} v+\mu \Delta v-\lambda v=|v|^{p-2} v,
\end{equation}
then $v$ satisfies
$P_{\mu}(v):=\!2{||\Delta v||}_2^2\!-\!\mu{||\nabla v||}_2^2\!-\!2{\gamma_p}{||v||}_p^p\!=\!0$.
\end{proposition}

The properties of $\mathcal{P}_{a, \mu}$ are related to the minimax structure of $E_{\mu}|_{S_a}$, and in particular to the behavior of $E_{\mu}$ with respect to dilations preserving the $L^2$-norm. To be more precise, for $u \in S_a$ and $s\in \mathbb{R}$, let
\begin{equation} \label{eq1.12}
(s \star u)(x) :=e^{\frac{N}{2} s} u\left(e^{s} x\right),~~~~\mbox{for}~~~~\mbox{a.e.}~~~~x \in \mathbb{R}^{N}.
\end{equation}
It results that $s \star u \in S_a$, and hence it is natural to study the fiber maps
\begin{equation} \label{eq1.13}
\Psi_{u}^{\mu}(s) :=E_{\mu}(s \star u)=\frac{e^{4s}}{2} {||\Delta u||}_2^2-\frac{\mu}{2}e^{2s} {||\nabla u||}_2^2-\frac{e^{2p\gamma_{p} s}}{p}{||u||}_p^p.
\end{equation}
We shall see that a critical point of $\Psi_{u}^{\mu}(s)$ allows to project a function on $\mathcal{P}_{a, \mu}$. Thus, monotonicity and convexity properties of $\Psi_{u}^{\mu}(s)$ strongly affects the structure of $\mathcal{P}_{a, \mu}$ (and in turn the geometry of $E_{\mu}|_{S_a}$ ), and also have a strong impact on properties of the the time-dependent equation (\ref{eq1.3}).
In this direction, let us consider the decomposition of $\mathcal{P}_{a, \mu}$ into the disjoint union $\mathcal{P}_{a, \mu}=\mathcal{P}_{+}^{a, \mu}\cup \mathcal{P}_{0}^{a, \mu} \cup \mathcal{P}_{-}^{a, \mu} $, where
$$\mathcal{P}_{+}^{a, \mu} :=\left\{u \in \mathcal{P}_{a, \mu} : 4||\Delta u||_{2}^{2}>\mu||\nabla u||_{2}^{2} + 2p \gamma_{p}^{2}||u||_{p}^{p}\right\}=\left\{u \in \mathcal{P}_{a, \mu} :\left(\Psi_{u}^{\mu}\right)^{\prime \prime}(0)>0\right\};$$
$$\mathcal{P}_{0}^{a, \mu} :=\left\{u \in \mathcal{P}_{a, \mu} : 4||\Delta u||_{2}^{2}=\mu||\nabla u||_{2}^{2} + 2p \gamma_{p}^{2}||u||_{p}^{p}\right\}=\left\{u \in \mathcal{P}_{a, \mu} :\left(\Psi_{u}^{\mu}\right)^{\prime \prime}(0)=0\right\};$$
$$\mathcal{P}_{-}^{a, \mu} :=\left\{u \in \mathcal{P}_{a, \mu} : 4||\Delta u||_{2}^{2}<\mu||\nabla u||_{2}^{2} + 2p \gamma_{p}^{2}||u||_{p}^{p}\right\}=\left\{u \in \mathcal{P}_{a, \mu} :\left(\Psi_{u}^{\mu}\right)^{\prime \prime}(0)<0\right\}.$$
For $u \in S_a$, $s \in \mathbb{R}$ and the fiber $\Psi_{u}^{\mu}$ introduced in (\ref{eq1.13}), we have
\begin{equation} \label{equa2.9}
\left(\Psi_{u}^{\mu}\right)^{\prime}(s)=2e^{4s}{||\Delta u||}_2^2-{\mu}e^{2s} {||\nabla u||}_2^2-2\gamma_{p} e^{2p\gamma_{p} s}{||u||}_p^p=P_{\mu}(s \star u),
\end{equation}
where $P_{\mu}$ is defined by (\ref{eq2.5}). From (\ref{equa2.9}), we can see immediately that:

\begin{corollary} \label{coroll2.1}
Let $u \in S_a$. Then: $s \in \mathbb{R}$ is a critical point for $\Psi_{u}^{\mu}$ if and only if $s \star u \in \mathcal{P}_{a, \mu}$.
\end{corollary}

In particular, $u \in \mathcal{P}_{a, \mu}$ if and only if $s=0$ is a critical point of $\Psi_{u}^{\mu}(s)$. For future convenience, we also recall that the map $(s, u) \in \mathbb{R} \times H^{2}(\R^N) \mapsto s \star u \in H^{2}(\R^N)$ is continuous (The proof is similar to that of Lemma 3.5 in \cite{TBns}). \\

To study the structure of the Pohozaev manifold $\mathcal{P}_{a,\mu}$, let us recall the constant
\begin{equation} \label{cost2.11}
\tilde{C}(N,p):=\!\frac{p}{2(p\gamma_p\!-\!1)C_{N,p}^p}{\Big( \frac{p\gamma_p\!-\!2}{p\gamma_p\!-\!1} \Big)}^{p\gamma_p\!-\!2}
\end{equation}
and the decomposition
$\mathcal{P}_{a, \mu}=\mathcal{P}_{+}^{a, \mu}\cup \mathcal{P}_{0}^{a, \mu} \cup \mathcal{P}_{-}^{a, \mu} $, we have:

\begin{lemma}\label{lem3.3}
Let $N\geq2$, $\overline{p}<p<4^*$, $a>0$, $\mu>0$ and ${\mu}^{p\gamma_p-2}{a}^{p-2}<\tilde{C}(N,p)$. Then $\mathcal{P}_{0}^{a, \mu}=\emptyset$, and $\mathcal{P}_{a, \mu}$ is a smooth manifold of codimension 2 in $H^{2}(\R^N)$.
\end{lemma}

\begin{proof}
Firstly, we claim that $\mathcal{P}_{0}^{a, \mu}=\emptyset$. Otherwise, there exists $u \in \mathcal{P}_{0}^{a, \mu}$ such that
$$
P_{\mu}(u)\!=\!2{||\Delta u||}_2^2\!-\!\mu{||\nabla u||}_2^2\!-\!2{\gamma_p}{||u||}_p^p\!=\!0,~~~~\Psi_{u}^{\prime \prime}(0)\!=\!8||\Delta u||_{2}^{2}\!-\!2\mu||\nabla u||_{2}^{2}\!-\!4p \gamma_{p}^{2}||u||_{p}^{p}\!=\!0.
$$
By elimination of ${||\nabla u||}_2^2$ and ${||u||}_p^p$, we have
$$
||\Delta u||_{2}^{2}=\gamma_{p}(p\gamma_{p}-1)||u||_{p}^{p}\leq \gamma_{p}(p\gamma_{p}-1)C_{N,p}^p a^{p(1-\gamma_p)} {||\Delta u||}_2^{p\gamma_p}
$$
and
$$
||\Delta u||_{2}^{2}=\frac{\mu(p\gamma_{p}-1)}{2(p\gamma_{p}-2)}{||\nabla u||}_2^2\leq\frac{a \mu(p\gamma_{p}-1)}{2(p\gamma_{p}-2)} ||\Delta u||_{2}.
$$
Then we deduce the lower and upper bounds of $||\Delta u||_{2}$ by
$$
 {\Big[ \frac{1}{\gamma_p(p\gamma_p-1)C_{N,p}^pa^{p(1-\gamma_p)}} \Big]}^{\frac{1}{p\gamma_p-2}} \leq ||\Delta u||_{2} \leq  \frac{a \mu(p\gamma_{p}-1)}{2(p\gamma_{p}-2)}.
$$
This leads to
$$ {\mu}^{p\gamma_p-2}{a}^{p-2}\geq \frac{1}{\gamma_p(p\gamma_p-1)C_{N,p}^p}{\Big( \frac{2(p\gamma_p-2)}{p\gamma_p-1} \Big)}^{p\gamma_p-2},$$
which contradicts with ${\mu}^{p\gamma_p-2}{a}^{p-2}<\tilde{C}(N,p)$.

Next we check that $\mathcal{P}_{a, \mu}$ is a smooth manifold of codimension 2 in $H^{2}(\R^N)$. Notice that $\mathcal{P}_{a, \mu}=\{u \in H^{2}(\R^N): P_{\mu}(u)=0, G(u)=0\}$ for $G(u)={||u||}_2^2-a^2$, with $P_{\mu}$ and $G$ of class $C^1$ in $H^{2}(\R^N)$. Thus, we have to
show that the differential $(d G(u), d P_{\mu}(u)) : H^{2}(\R^N) \rightarrow \mathbb{R}^{2}$ is surjective, for every $u \in \mathcal{P}_{a, \mu}$. We need a claim: $\forall u \in \mathcal{P}_{a, \mu} $, there exists $\varphi \in T_{u}S_a$ such that $dP_{\mu}(u) [ \varphi ] \neq 0$. Once that the existence of $\varphi$ is established, the system
$$
\left\{\begin{array}{l}{d G(u)[\alpha \varphi+\beta u]=x} \\ {d P_{\mu}(u)[\alpha \varphi+\beta u]=y}\end{array} \Leftrightarrow\left\{\begin{array}{l}{2\beta a^{2}=x} \\ {\alpha d P_{\mu}(u)[\varphi]+\beta d P_{\mu}(u)[u]=y}\end{array}\right.\right.
$$
is solvable with respect to $\alpha, \beta$, for every $(x, y) \in{\R}^2$, and hence the surjectivity is proved.

Now, suppose by contradiction that
$$\exists~~ u \in \mathcal{P}_{a, \mu}~~~~\mbox{such that}~~~~dP_{\mu}(u) [ \varphi ]=0~~~~\mbox{for any}~~~~\varphi \in T_{u}S_a.$$
Then $u$ is a constrained critical point for the functional $P_{\mu}$ on $S_a$, and hence by the Lagrange multipliers rule there exists $\nu \in \mathbb{R}$ such that
$$  4{\Delta}^{2}u+2\mu \Delta u-\nu u=2p\gamma_p{|u|}^{p-2}u~~~~\mbox{in}~~~~\mathbb{R}^{N}. $$
But, by Proposition \ref{pro3.01}, this implies that $8||\Delta u||_{2}^{2}-2\mu||\nabla u||_{2}^{2}-4p \gamma_{p}^{2}||u||_{p}^{p}=0$,
that is $u \in \mathcal{P}_{0}^{a, \mu}$, a contradiction.  \\
\end{proof}

The manifold $\mathcal{P}_{a, \mu}$ is then divided into its two components $\mathcal{P}_{+}^{a, \mu}$ and $\mathcal{P}_{-}^{a, \mu}$, having disjoint closure. We can prove that $\mathcal{P}_{a, \mu}$ is a natural constraint, in the following sense:
\begin{lemma}\label{lemma1.3}
Let $N\geq2$, $\overline{p}<p<4^*$, $a>0$, $\mu>0$ and ${\mu}^{p\gamma_p-2}{a}^{p-2}<\tilde{C}(N,p)$.
If $u\in \mathcal{P}_{a, \mu}$ is a critical point for $E_{\mu}|_{\mathcal{P}_{a, \mu}}$, then $u$ is a critical point for $E_{\mu}|_{S_a}$. 
\end{lemma}

\begin{proof}
We recall that by Lemma \ref{lem3.3}, $\mathcal{P}_{a, \mu}$ is a smooth manifold of codimension $2$ in $H^2$, and its subset $\mathcal{P}_0^{a, \mu}$ is empty. If $u \in \mathcal{P}_{a, \mu}$ is a critical point for $E_{\mu}|_{\mathcal{P}_{a, \mu}}$, then by the Lagrange multipliers rule there exists $\lambda, \nu \in \mathbb{R}$ such that
$$
d E_{\mu}(u)[\varphi]-\lambda \int_{\mathbb{R}^{N}} u {\varphi}-\nu d P_{\mu}(u)[\varphi]=0,~~~~\forall \varphi \in H^2.
$$
That is $(1-4\nu){\Delta}^{2}u+\mu(1-2\nu)\Delta u-\lambda u+(2p\gamma_p\nu-1){|u|}^{p-2}u=0$ in $\mathbb{R}^{N}$. But, by the Pohozaev identity Proposition \ref{pro3.01}, this implies that
$$
2(1-4\nu)||\Delta u||_{2}^{2}-\mu(1-2\nu)||\nabla u||_{2}^{2}+2\gamma_p(2p\gamma_p\nu-1)||u||_{p}^{p}=0.
$$
Since $u \in \mathcal{P}_{a, \mu}$, we have $\nu(4||\Delta u||_{2}^{2}-\mu||\nabla u||_{2}^{2}-2p{\gamma_p^2}||u||_{p}^{p})=0$.
But the term inside the bracket cannot be $0$, since $u \notin \mathcal{P}_{0}^{a, \mu}$, and then necessarily $\nu=0$.\\
\end{proof}

Next, we study the fiber maps $\Psi_{u}^{\mu}(s)$ and determine the location and types of critical points for $E_{\mu}|_{S_a}$.
Let us consider the constrained functional $E_{\mu}|_{S_{a}}$. From (\ref{eq2.3}), we have
\begin{equation*}
 E_{\mu}(u)\geq \frac{1}{2}{||\Delta u||}_2^2-\frac{\mu a}{2}{||\Delta u||}_2-\frac{C_{N,p}^p}{p}a^{p(1-\gamma_p)}{||\Delta u||}_2^{p\gamma_p}, ~~~~\forall u\in S_a.
\end{equation*}
Therefore, to understand the geometry of the functional $E_{\mu}|_{S_{a}}$ it is useful to consider the function $h : \mathbb{R}^{+} \rightarrow \mathbb{R}$:
$$h(t)=\frac{1}{2}t^2-\frac{\mu a}{2}t-\frac{C_{N,p}^p}{p}a^{p(1-\gamma_p)}t^{p\gamma_p}.$$
Since $\mu\!>\!0$ and $p\gamma_p\!>\!2$ for $\overline{p}\!<\!p\!<\!4^*$, we have that $h(0^+)\!=\!0^{-}$ and $h(+\infty)\!=\!-\infty $. 

\begin{lemma}\label{lem3.1}
Let $N\!\geq\!2$, $\overline{p}\!<\!p\!<\!4^*$, $a,\mu\!>\!0$ and
${\mu}^{p\gamma_p-2}{a}^{p-2}\!<\!\tilde{C}(N,p)$. Then the function $h$ has a local strict minimum at negative level and a global strict maximum at positive level. Moreover, there exist $R_0$ and $R_1$, both depending
on $a$ and $\mu$, such that
$$0 <\mu a< R_0 <\bar{t}< R_1,~~~~~~~~h(R_0)=0=h(R_1)$$
and $h(t)\!>\!0$ if only if $t\!\in\! (R_0,R_1)$. Here $\bar{t}\!=\!{\Big[ \frac{p}{2(p\gamma_p-1)C_{N,p}^pa^{p(1-\gamma_p)}} \Big]}^{\frac{1}{p\gamma_p-2}}$. 
\end{lemma}

\begin{proof}
For $t>0$, it is easy to see that $h(t)>0$ if and only if
$$ \varphi(t)>\frac{\mu a}{2},~~~~\mbox{with}~~~~\varphi(t)=\frac{1}{2}t
-\frac{C_{N,p}^p}{p}a^{p(1-\gamma_p)}t^{p\gamma_p-1}.$$
Also, $\varphi$ has a unique critical point on $(0,+\infty)$, which is a global maximum point at positive level,
in $\bar{t}={\Big[ \frac{p}{2(p\gamma_p-1)C_{N,p}^pa^{p(1-\gamma_p)}} \Big]}^{\frac{1}{p\gamma_p-2}}$,
and the maximum level is $\varphi(\overline{t})=\frac{p\gamma_p-2}{2(p\gamma_p-1)}\overline{t}$.

Notice that
$$ {\mu}^{p\gamma_p-2}{a}^{p-2}<\tilde{C}(N,p) \Longleftrightarrow \varphi(\overline{t})>\frac{\mu a}{2}.$$
Therefore, $h$ is positive on an open interval $(R_0,R_1)$ if only if ${\mu}^{p\gamma_p-2}{a}^{p-2}<\tilde{C}(N,p)$. It follows immediately that $h$ has a global maximum at positive level in $(R_0,R_1)$. Moreover, since $h(0^+)=0^{-}$, there exists a local minimum point at negative level in $(0,R_0)$. We also observe that $0 <\mu a< R_0 <\bar{t}< R_1$. The fact that $h$ has
no other critical points can be verified observing that $h'(t)=0$ if only if
$$
\psi(t)=\frac{\mu a}{2}, ~~~~\mbox{with}~~~~\psi(t)=t-
{\gamma_p}{C_{N,p}^p}{a^{p(1-\gamma_p)}}t^{p\gamma_p-1}.
$$
Clearly $\psi$ has only one critical point at $\hat{t}={\Big[ \frac{1}{\gamma_p(p\gamma_p-1)C_{N,p}^pa^{p(1-\gamma_p)}} \Big]}^{\frac{1}{p\gamma_p-2}}$, which is a strict maximum and
$\psi(\hat{t})=\frac{p\gamma_p-2}{p\gamma_p-1}\hat{t}$. 
Moreover, we have
\begin{equation} \label{eq3.2}
\psi(\hat{t})>\frac{\mu a}{2} \Longleftrightarrow {\mu}^{p\gamma_p-2}{a}^{p-2}<\frac{1}{\gamma_p(p\gamma_p-1)C_{N,p}^p}{\Big( \frac{2(p\gamma_p-2)}{p\gamma_p-1} \Big)}^{p\gamma_p-2}.
\end{equation}
From $p\gamma_p<2^{p\gamma_p-1}$ and ${\mu}^{p\gamma_p-2}{a}^{p-2}<\tilde{C}(N,p)$, we can check that (\ref{eq3.2}) holds.
\end{proof}


\begin{lemma}\label{lem3.4}
Let $N\!\geq\!2$, $\overline{p}\!<\!p\!<\!4^*$, $a,\mu\!>\!0$ and ${\mu}^{p\gamma_p-2}{a}^{p-2}\!<\!\tilde{C}(N,p)$.
For every $u \!\in\! S_{a}$, the function $\Psi_{u}^{\mu}$ has exactly two critical points $s_{u}\!<\!t_{u} \!\in\! \mathbb{R}$ and two zeros $c_{u}\!<\!d_{u} \!\in\! \mathbb{R}$, with $s_{u}\!<\!c_{u}\!<\!t_{u}\!<\!d_{u}$. Moreover:\\
$(1)$ $s_{u} \star u \in \mathcal{P}_{+}^{a,\mu}$ and $t_{u} \star u \in \mathcal{P}_{-}^{a,\mu}$, and if $s \star u \in \mathcal{P}_{a,\mu}$, then either $s=s_{u}$ or $s=t_{u}$; \\
$(2)$ $||\Delta(s \star u)||_{2} \leq R_{0}$ for every $s \leq c_{u}$, and
$$E_{\mu}\left(s_{u} \star u\right)=\min \left\{E_{\mu}(s \star u) : s \in \mathbb{R} \text { and }||\Delta(s \star u)||_{2}<R_{0}\right\}<0.$$
$(3)$ We have
$$ E_{\mu}\left(t_{u} \star u\right)=\max \{E_{\mu}(s \star u) : s \in \mathbb{R}\}>0,$$
and $\Psi^{\mu}_{u}$ is strictly decreasing and concave on $\left(t_{u},+\infty\right)$.  \\
$(4)$ The maps $u \in S_{a} \mapsto s_{u} \in \mathbb{R}$ and $u \in S_{a} \mapsto t_{u} \in \mathbb{R}$ are of class $C^1$.\\
\end{lemma}

\begin{proof}
Let $u\!\in\! S_a$, as observed in Corollary \ref{coroll2.1}, $s\star u \!\in\! \mathcal{P}_{a,\mu}$ if and only if $(\Psi_{u}^{\mu})^{\prime}(s)\!=\!0$. Thus,
we first show that $\Psi_{u}^{\mu}$ has at least two critical points. To this end, we recall that by (\ref{eq2.3})
$$\Psi_{u}^{\mu}(s)=E_{\mu}(s \star u) \geq h\left(||\Delta(s \star u)||_{2}\right)=h\left(e^{2s}||\Delta u||_{2}\right).$$
Thus, the $C^2$ function $\Psi_{u}^{\mu}$ is positive on $\left(\frac{1}{2}\log \frac{R_{0}}{||\Delta u||_{2}}, \frac{1}{2}\log \frac{R_{1}}{||\Delta u||_{2}}\right)$, and clearly $\Psi_{u}^{\mu}(-\infty)=0^{-}$, $\Psi_{u}^{\mu}(+\infty)=-\infty$. It follows that $\Psi_{u}^{\mu}$ has at least two critical points $s_u<t_u$, with $s_u$ local minimum point on $(-\infty, \frac{1}{2}\log \frac{R_{0}}{||\Delta u||_{2}})$ at negative level, and $t_u > s_u$ global maximum point at positive level.
It is not difficult to check that there are no other critical points. Indeed,  $(\Psi_{u}^{\mu})^{\prime}(s)=0$ reads
$$\varphi(s)=\mu||\nabla u|||_{2}^{2},~~~~\mbox{with}~~~~ \varphi(s)=2e^{2s}||\Delta u||_{2}^{2}-2\gamma_{p}e^{2(p\gamma_{p} -1) s}||u||_{p}^{p}. $$
But $\varphi$ has a unique maximum point at $\bar{s}$ with $ e^{\bar{s}}=\Big[ \frac{||\Delta u||_{2}^{2}}{   {\gamma_{p}}(p\gamma_{p}-1) ||u||_{p}^{p}  } \Big]^{\frac{1}{2(p\gamma_{p}-2)}}$. By the Gagliardo-Nirenberg inequality \eqref{2.9} and ${\mu}^{p\gamma_p-2}{a}^{p-2}<\tilde{C}(N,p)$, we deduce that
$$
\varphi(\bar{s}) \geq \frac{2(p\gamma_p-2)}{p\gamma_p-1} {\Big[ \frac{1}{\gamma_p(p\gamma_p-1)C_{N,p}^pa^{p(1-\gamma_p)}} \Big]}^{\frac{1}{p\gamma_p-2}} ||\Delta u||_{2} >\mu a||\Delta u||_{2}\geq \mu||\nabla u|||_{2}^{2}.
$$
That is $\varphi(\bar{s})\!>\!\mu||\nabla u|||_{2}^{2}$, so $\Psi_{u}^{\mu}$ has exactly two critical points. By Corollary \ref{coroll2.1}, we have $s_{u} \star u$, $t_{u} \star u \!\in\! \mathcal{P}_{a,\mu}$ and $s \star u \!\in\! \mathcal{P}_{a,\mu}$ implies $s \!\in\!\left\{s_{u}, t_{u}\right\}$. By minimality $(\Psi_{s_{u} \star u}^{\mu})^{\prime \prime}(0)\!=\!(\Psi_{u}^{\mu})^{\prime \prime}\left(s_{u}\right) \!\geq \!0$, and in fact strict inequality must hold, since $\mathcal{P}_{0}^{a,\mu}=\emptyset$; namely $s_{u} \star u \in \mathcal{P}_{+}^{a,\mu}$. In the same way $t_{u} \star u \in \mathcal{P}_{-}^{a,\mu}$.

By monotonicity and recalling the behavior at infinity, $\Psi_{u}^{\mu}$ has moreover exactly two zeros $c_{u}\!<\!d_{u}$, with $s_{u}\!<\!c_{u}\!<\!t_{u}\!<\!d_{u}$. Being a $C^2$ function, $\Psi_{u}^{\mu}$ has at least two inflection points. Arguing as before, we can easily check that $\Psi_{u}^{\mu}$ has exactly two inflection points. In particular, $\Psi_{u}^{\mu}$ is concave on $[t_u,+\infty)$.

It remains to show that $u \mapsto s_{u}$ and $u \mapsto t_{u}$ are of class $C^1$. To this end, we apply the implicit function theorem on the $C^1$ function $\Phi(s, u) :=(\Psi_{u}^{\mu})^{\prime}(s)$. We use $\Phi\left(s_{u}, u\right)=0$, $\partial_{s} \Phi\left(s_{u}, u\right)=(\Psi_{u}^{\mu})^{\prime \prime}\left(s_{u}\right)>0$, and the fact that it is not possible to pass with continuity from $\mathcal{P}_{+}^{a,\mu}$ to $\mathcal{P}_{-}^{a,\mu}$ (since $\mathcal{P}_{0}^{a,\mu}=\emptyset$). The same argument proves that $u \mapsto t_{u}$ is $C^{1}$.

\end{proof}

Recall that $S_{a,r}=S_a\cap H^2_{rad}$. For $k>0$, let us set
\begin{equation} \label{equ2.13}
A_{k} :=\left\{u \in S_a :||\Delta u||_{2}<k\right\},~~~~~~~~A_{k}^r :=\left\{u \in S_{a,r}  :||\Delta u||_{2}<k\right\},
\end{equation}
and introduce two minimization problems
\begin{equation} \label{equ2.14}
 m(a,\mu):=\inf _{u \in A_{R_{0}}} E_{\mu}(u),~~~~~~~~m_r(a,\mu):=\inf _{u \in A_{R_{0}^r}} E_{\mu}(u).
\end{equation}

\begin{corollary} \label{coro3.5}
Let $N\!\geq\!2$, $\overline{p}\!<\!p\!<\!4^*$, $a,\mu\!>\!0$ and ${\mu}^{p\gamma_p-2}{a}^{p-2}\!<\!\tilde{C}(N,p)$. Then, we have  $$\mathcal{P}_{+}^{a,\mu} \!\subset \! A_{R_{0}}\!:=\!\left\{u \in S_a :||\Delta u||_{2}<R_{0}\right\}~~~~\mbox{and}~~~~\sup _{\mathcal{P}_{+}^{a,\mu}} E_{\mu} \!\leq\! 0 \!\leq\!\inf _{\mathcal{P}_{-}^{a,\mu}} E_{\mu}. $$
\end{corollary}

\begin{proof}
It is a direct conclusion of Lemma \ref{lem3.4}. Indeed, $\forall u \!\in\! \mathcal{P}_{+}^{a,\mu}$, Lemma \ref{lem3.4} implies that $s_{u}\!=\!0$, $E_{\mu}(u)\!\leq\!0$ and $||\Delta u||_{2}\!<\!R_{0}$. Similarly, $u \!\in\! \mathcal{P}_{-}^{a,\mu}$ implies that $t_{u}\!=\!0$ and $E_{\mu}(u)\!\geq\!0$.
\end{proof}

\begin{lemma} \label{lem3.6}
Let $N\!\geq\!2$, $\overline{p}\!<\!p\!<\!4^*$, $a,\mu\!>\!0$ and ${\mu}^{p\gamma_p-2}{a}^{p-2}\!<\!\tilde{C}(N,p)$. It results that
$$m(a, \mu) \in(-\infty, 0),~~~~~~~~m(a, \mu)=\inf _{\mathcal{P}_{a,\mu}} E_{\mu}=\inf _{\mathcal{P}_{+}^{a,\mu}} E_{\mu},$$
and
$$m_r(a, \mu) \in(-\infty, 0),~~~~~~~~m_r(a, \mu)=\inf _{\mathcal{P}_{a,\mu}\cap S_{a,r}} E_{\mu}=\inf _{\mathcal{P}_{+}^{a,\mu}\cap S_{a,r}} E_{\mu}.$$
Moreover, there exists a constant $\rho>0$ small enough such that
$$
m(a, \mu)<\min\{ \inf _{ \overline{{A_{\rho}}} } E_{\mu},~~~~\inf _{ \overline{{A_{R_{0}}} } \setminus A_{R_{0}-\rho}} E_{\mu}  \},
~~~~~~~~m_r(a, \mu)<\min\{ \inf _{ \overline{{A^r_{\rho}}} } E_{\mu},~~~~\inf _{ \overline{{A^r_{R_{0}}} } \setminus A^r_{R_{0}-\rho}} E_{\mu}  \}.
$$
\end{lemma}

\begin{proof}
We only prove the Lemma for $m(a, \mu)$, since we can similarly prove the case of $m_r(a, \mu)$. For $u \in A_{R_{0}}$, we have
$$
E_{\mu}(u) \geq h\left(||\Delta u||_{2}\right) \geq \min _{t \in\left[0, R_{0}\right]} h(t)>-\infty,
$$
and hence $m(a, \mu)>-\infty$. Moreover, for any $u\in S_a$ we have $||\Delta(s \star u)||_{2}<R_{0}$ and $E_{\mu}(s \star u)<0$ for $s \ll-1$, and hence $m(a, \mu)<0$.

By Corollary \ref{coro3.5}, we have $m(a, \mu) \leq \inf _{\mathcal{P}_{+}^{a,\mu}} E_{\mu}$ since $\mathcal{P}_{+}^{a,\mu} \subset A_{R_{0}}$. On the other hand, if $u \in A_{R_{0}}$ then
$s_{u} \star u \in \mathcal{P}_{+}^{a,\mu} \subset A_{R_{0}}$, and
$$
E_{\mu}\left(s_{u} \star u\right)=\min \left\{E_{\mu}(s \star u) : s \in \mathbb{R} \text { and }||\Delta(s \star u)||_{2}<R_{0}\right\} \leq E_{\mu}(u).
$$
which implies that $\inf _{\mathcal{P}_{+}^{a,\mu}} E_{\mu} \leq m(a, \mu)$. To prove that $\inf _{\mathcal{P}_{+}^{a,\mu}} E_{\mu}=\inf _{\mathcal{P}_{a,\mu}} E_{\mu}$, it is sufficient to recall that $E_{\mu}\geq0$ on $\mathcal{P}_{-}^{a,\mu}$, see Corollary \ref{coro3.5}.

Finally, by continuity of $h$ there exists $\rho> 0$ such that $h(t)\geq \frac{m(a, \mu)}{2}$ if $t \in [0, \rho]\cup [R_0- \rho, R_0]$.
Therefore
$$
E_{\mu}(u) \geq h\left(||\Delta u||_{2}\right) \geq \frac{m(a, \mu)}{2}>m(a, \mu)
$$
for every $u\in S_a$ with $ ||\Delta u||_{2} \in [0, \rho]\cup [R_0- \rho, R_0]$.
\end{proof}

\begin{lemma} \label{lem3.7}
Let $N\!\geq\!2$, $\overline{p}\!<\!p\!<\!4^*$, $a,\mu\!>\!0$ and ${\mu}^{p\gamma_p-2}{a}^{p-2}\!<\!\tilde{C}(N,p)$.
Suppose that $u\!\in\! S_a$ and $E_{\mu}(u)\!<\!m(a, \mu)$, or $u\!\in\! S_{a,r}$ and $E_{\mu}(u)\!<\!m_r(a, \mu)$. Then the value $t_u$ defined by Lemma \ref{lem3.4} is negative.
\end{lemma}

\begin{proof}
We only prove the Lemma for $m(a, \mu)$, since the case of $m_r(a,\mu)$ is analogous. Consider $\Psi^{\mu}_{u}$ and $s_{u}<c_{u}<t_{u}<d_{u}$ stated in Lemma \ref{lem3.4}. If $d_{u} \leq 0$, then $t_{u}<0$, and hence we can assume by contradiction that $d_u>0$. If $0\in(c_{u}, d_{u})$, then $E_{\mu}(u)=\Psi^{\mu}_{u}(0)>0$, which is not possible since $E_{\mu}(u)<m(a, \mu)<0$. Therefore $c_{u}>0$, and by
Lemma \ref{lem3.4}-(2)
\begin{align*}
m(a, \mu) &>E_{\mu}(u)=\Psi_{u}^{\mu}(0) \geq \inf _{s \in\left(-\infty, c_{u}\right]} \Psi_{u}^{\mu}(s) \\ & \geq \inf \left\{E_{\mu}(s \star u) : s \in \mathbb{R} \text { and }||\Delta(s \star u)||_{2}<R_{0}\right\}=E_{\mu}\left(s_{u} \star u\right) \geq m(a, \mu), \end{align*}
which is again a contradiction.
\end{proof}

\begin{lemma} \label{lem3.8}
Let $N\!\geq\!2$, $\overline{p}\!<\!p\!<\!4^*$, $a,\mu\!>\!0$ and ${\mu}^{p\gamma_p-2}{a}^{p-2}\!<\!\tilde{C}(N,p)$. It results that
$$\tilde{\sigma}(a, \mu) :=\inf _{u \in \mathcal{P}_{-}^{a,\mu}} E_{\mu}(u)>0.$$
\end{lemma}

\begin{proof}
Let $t_{max}$ denote the strict maximum of the function $h$ at positive level, see Lemma \ref{lem3.1}. For every $u \in \mathcal{P}_{-}^{a,\mu}$, there exists $\tau_{u} \in \mathbb{R}$ such that $||\Delta\left(\tau_{u} \star u\right)||_{2}=t_{\max }$. Moreover, since $u \in \mathcal{P}_{-}^{a,\mu}$ we also have by Lemma \ref{lem3.4} that the value $0$ is the unique strict maximum of the function $\Psi_{u}^{\mu}$. Therefore
$$
E_{\mu}(u)=\Psi_{u}^{\mu}(0) \geq \Psi_{u}^{\mu}\left(\tau_{u}\right)=E_{\mu}\left(\tau_{u} \star u\right) \geq h\left(||\Delta\left(\tau_{u} \star u\right)||_{2}\right)=h\left(t_{\max }\right)>0.
$$
The arbitrariness of $u \in \mathcal{P}_{-}^{a,\mu}$ implies that $\inf_{\mathcal{P}_{-}^{a,\mu}} E_{\mu} \geq \max _{\mathbb{R}} h>0,$ as desired.
\end{proof}

\section{The exact upper bounds of $m(a,\mu)$ and $m_r(a,\mu)$ }
In this section, we give the exact upper bounds of $m(a,\mu)$ and $m_r(a,\mu)$, which are vital in compactness analysis of the related Palais-Smale sequences in the forthcoming section.

We observe that for any $ k \in [R_{0}, R_{1}]$, $m(a, \mu)=\inf _{A_{R_{0}}} E_{\mu}$ can be relaxed to
\begin{equation} \label{eqn3.1}
m(a, \mu)=\inf _{u\in A_{k}} E_{\mu}(u),~~~~~~~~ \mbox{where}~~~~A_{k}:=\left\{ u \in S_a,||\Delta u||_{2}<k\right\}.
\end{equation}
Indeed, if $||\Delta u||_{2}\!\in\! [R_{0}, R_{1}]$, then $E_{\mu}(u) \!\geq\! h(||\Delta u||_{2})\!\geq\!0\!>\!\inf _{A_{R_{0}}} E_{\mu}$, see (\ref{eq2.3}) and Lemma \ref{lem3.1}. So we have $m(a, \mu)\!=\!\inf _{A_{\bar{t}}} E_{\mu}$ with $\bar{t}\!=\!{\Big[ \frac{p}{2(p\gamma_p-1)C_{N,p}^pa^{p(1-\gamma_p)}} \Big]}^{\frac{1}{p\gamma_p-2}}$ defined by Lemma \ref{lem3.1}.

Now, we transform $m(a, \mu)=\inf _{A_{\bar{t}}} E_{\mu}$ into another equivalent constrained minimization problem. For any $u\in A_{\bar{t}}$, we know that $||u||_{2}=a$ and $||\Delta u||_{2}<\bar{t}$, let $v(x):=\tilde{b}u(\tilde{a}x)$ with
$$
\tilde{a}=\left(\frac{2}{\mu}\right)^{\frac{1}{2}},~~~~~~~~
\tilde{b}=\left(\frac{8}{\mu^{2}}\right)^{\frac{1}{p-2}},
~~~~~~~~\tilde{c}=\left(\frac{8}{\mu^{2}}\right)
^{\frac{2}{p-2}}\left(\frac{\mu}{2}\right)^{\frac{N}{2}}=
2^{\frac{6}{p-2}-\frac{N}{2}} \mu^{\frac{N}{2}-\frac{4}{p-2}},
$$
then we have $||v||_2^2=\tilde{c}||u||_2^2=\tilde{c} a^2$, $||\Delta v||_2={\tilde{a}}^{2-\frac{N}{2}} \tilde{b} ||\Delta u||_2<{\tilde{a}}^{2-\frac{N}{2}} \tilde{b} \bar{t}$ and
$$E_{\mu}(u)={\tilde{a}}^N {\tilde{b}}^{-p}\Phi_{0}(v),~~~~~~~~\mbox{where}~~~~\Phi_{0}(v)={||\Delta v||}_2^2-{2}{||\nabla v||}_2^2-\frac{1}{p}{||v||}_p^p.$$
Similar to (\ref{eq2.3}), we have
\begin{equation*}
 \Phi_{0}(v) \geq {||\Delta v||}_2^2-2 a \sqrt{\tilde{c}} {||\Delta v||}_2-\frac{C_{N,p}^p}{p}(a \sqrt{\tilde{c}})^{p(1-\gamma_p)}{||\Delta v||}_2^{p\gamma_p}, ~~~~\forall v \in S_{a \sqrt{\tilde{c}}}.
\end{equation*}
To study the geometry of the functional $\Phi_{0}|_{S_{a \sqrt{\tilde{c}}}}$, we consider the function $\tilde{h} : \mathbb{R}^{+} \rightarrow \mathbb{R}$:
$$\tilde{h}(\tau)=\tau^2-2 a \sqrt{\tilde{c}} \tau-\frac{C_{N,p}^p}{p}(a \sqrt{\tilde{c}})^{p(1-\gamma_p)}\tau^{p\gamma_p}.$$


\begin{lemma}\label{lem3.4.0}
Let $N\!\geq\!2$, $\overline{p}\!<\!p\!<\!4^*$, $a,\mu\!>\!0$ and
${\mu}^{p\gamma_p-2}{a}^{p-2}\!<\!\tilde{C}(N,p)$. Then the function $\tilde{h}$ has a local strict minimum at negative level and a global strict maximum at positive level. Moreover, there exist $\tilde{R}_0$ and $\tilde{R}_1$, both depending
on $a$ and $\mu$, such that
$$0 <2 a \sqrt{\tilde{c}}< \tilde{R}_0 <\tilde{\tau}< \tilde{R}_1,~~~~~~~~\tilde{h}(\tilde{R}_0)=0=\tilde{h}(\tilde{R}_1)$$
and $\tilde{h}(\tau)\!>\!0$ if only if $\tau\!\in\! (\tilde{R}_0,\tilde{R}_1)$. Here $\tilde{\tau}\!=\!{\Big[ \frac{p}{(p\gamma_p\!-\!1)C_{N,p}^p(a \sqrt{\tilde{c}})^{p(1-\gamma_p)}} \Big]}^{\frac{1}{p\gamma_p-2}}$. 
\end{lemma}

\begin{proof}
For $\tau>0$, it is easy to see that $\tilde{h}(\tau)>0$ if and only if
$$ \tilde{\varphi}(\tau)>2 a \sqrt{\tilde{c}},~~~~\mbox{with}~~~~\tilde{\varphi}(\tau)=\tau-\frac{C_{N,p}^p}{p}(a \sqrt{\tilde{c}})^{p(1-\gamma_p)}\tau^{{p\gamma_p}-1}.$$
Also, $\tilde{\varphi}$ has a unique critical point on $(0,+\infty)$, which is a global maximum point at positive level,
in $\tilde{\tau}={\Big[ \frac{p}{(p\gamma_p-1)C_{N,p}^p(a \sqrt{\tilde{c}})^{p(1-\gamma_p)}} \Big]}^{\frac{1}{p\gamma_p-2}}$,
and the maximum level is $ \tilde{\varphi}(\tilde{\tau})=\frac{p\gamma_p-2}{p\gamma_p-1}\tilde{\tau}$.

Notice that
$$ {\mu}^{p\gamma_p-2}{a}^{p-2}<\tilde{C}(N,p) \Longleftrightarrow \tilde{\varphi}(\tilde{\tau})>2 a \sqrt{\tilde{c}}.$$
Therefore, $\tilde{h}$ is positive on an open interval $(\tilde{R}_0,\tilde{R}_1)$ if only if ${\mu}^{p\gamma_p-2}{a}^{p-2}<\tilde{C}(N,p)$. It follows immediately that $\tilde{h}$ has a global maximum at positive level in $(\tilde{R}_0,\tilde{R}_1)$. Moreover, since $\tilde{h}(0^+)=0^{-}$, there exists a local minimum point at negative level in $(0,\tilde{R}_0)$. We also observe that $0 <2 a \sqrt{\tilde{c}}< \tilde{R}_0 <\tilde{\tau}< \tilde{R}_1$. The fact that $\tilde{h}$ has no other critical points can be verified as in the proof of Lemma \ref{lem3.1}. \\
\end{proof}

Define
$$m_{\Phi_0}(a, \mu)=\inf_{v\in\tilde{A}_{\tilde{R}_0}} \Phi_{0}(v),~~~~~~\mbox{where}~~~~\tilde{A}_{\tilde{R}_0}:=\left\{ v \in S_{a \sqrt{\tilde{c}}},||\Delta v||_{2}<\tilde{R}_0\right\}.$$
By Lemma \ref{lem3.4.0}, we have $m_{\Phi_0}(a, \mu)=\inf_{v\in\tilde{A}_{\tilde{\tau}}} \Phi_{0}(v)$ with $\tilde{\tau}={\Big[ \frac{p}{(p\gamma_p-1)C_{N,p}^p(a \sqrt{\tilde{c}})^{p(1-\gamma_p)}} \Big]}^{\frac{1}{p\gamma_p-2}}$. It is easy to prove that
$$  \inf\big\{E_{\mu}(u) : u\in S_{a},||\Delta u||_{2}<\bar{t}\big\} ={\tilde{a}}^N {\tilde{b}}^{-p} \inf \big\{ \Phi_{0}(v) :v \in S_{a \sqrt{\tilde{c}}}, ||\Delta v||_{2}<{\tilde{a}}^{2-\frac{N}{2}} \tilde{b} \bar{t}\big\}.  $$
A direct calculation implies that $\tilde{\tau}={\tilde{a}}^{2-\frac{N}{2}} \tilde{b} \bar{t}$, therefore we have
\begin{equation} \label{qan3.2.3}
 m(a, \mu)=\inf _{u\in A_{\bar{t}}} E_{\mu}(u)={\tilde{a}}^N {\tilde{b}}^{-p}\inf_{v\in\tilde{A}_{\tilde{\tau}}} \Phi_{0}(v)={\tilde{a}}^N {\tilde{b}}^{-p}m_{\Phi_0}(a, \mu).
\end{equation}

\begin{lemma} \label{lema3.4.1}
Let $N \!\geq\! 2$ and $\overline{p}\!<\!p\!<\!4^*$, $a,\mu\!>\!0$ such that
${\mu}^{p\gamma_p-2}{a}^{p-2}\!<\!\tilde{C}(N,p)$. If $N \!\geq\! 5$ and $p<4$, then $m_{\Phi_0}(a, \mu)\!<\!-\tilde{c} a^2$ and $m(a, \mu)\!<\!-\frac{a^2\mu^2}{8}$, furthermore, we have $m_r(a,\mu)\!<\!-\frac{a^2\mu^2}{8}$.
\end{lemma}
\begin{proof}
This is motivated by Lemma 5.5 in \cite{Nbal}. We give out the details since we consider a local minimization problem. From Appendix B.4 in \cite{LgOs}, it is not difficult to check that $\psi(x)=|x|^{-\frac{N-2}{2}} J_{\frac{N-2}{2}}(|x|)$ satisfies
$$(\Delta+1) \psi=0, ~~~~~~~~\mbox{in}~~~~\mathbb{R}^{N}.$$
Here $J_{\nu}$ is the Bessel function of the first kind with order $\nu$. Then, for all $m\in \mathbb{N}$, we define
$$\psi_{m}(x)=\psi(x) \phi\left(\frac{x}{m}\right),$$
where $\phi \in C^{\infty}(\R^N)$ is such that $\phi(x)=1$ if $|x|\leq1$, $\phi(x)=0$ if $|x|\geq2$ and $0\leq\phi(x)\leq1$ for all $x \in\R^N$. Using the fact that $(\Delta+1)\psi=0$, we have
$$\Delta \psi_{m}(x)=-\psi(x)\phi\left(\frac{x}{m}\right)+\frac{1}{m^{2}} \psi(x)\Delta \phi\left(\frac{x}{m}\right) +\frac{2}{m} \nabla \phi\left(\frac{x}{m}\right) \cdot \nabla \psi(x)$$
and $(\Delta+1) \psi_{m}(x)=\frac{1}{m^{2}} \psi(x)\Delta \phi\left(\frac{x}{m}\right) +\frac{2}{m} \nabla \phi\left(\frac{x}{m}\right) \cdot \nabla \psi(x)$. As in the proof of Lemma 5.5 in \cite{Nbal}, there exists $m_1\!>\!0$ large such that, for $m\!\geq\!m_1$, we have
\begin{align*}
C_1\!\leq\! ||\psi_{m}||_{2}^2\!\leq\! C_2 m,~~||\psi(x)\Delta\phi\left(\frac{x}{m}\right)||_{2}^2\!\leq\! C_3 m,
~~||\psi_{m}||_{p}^p \!\geq\! C_4,~~||\nabla \phi\left(\frac{x}{m}\right) \cdot \nabla \psi(x)||_{2}^2\!\leq\! C_5 m,
\end{align*}
consequently, there exists $m_2>0$ large such that, for $m\geq max\{m_1,m_2\}$, we have
\begin{align*}
||\Delta \psi_{m}||_{2}^2\!
\leq 3||\psi_{m}||_{2}^2\!+\!C_6 m^{-1}~~~~\mbox{and}~~~~||(\Delta+1) \psi_{m}||_{2}^2\leq C_7 m^{-1},
\end{align*}
where $C_i(i\in\mathbb{N})$ are some positive constants independent of $m$.

For any $m \!\in\! \mathbb{N}$, let $\widetilde{\psi}_{m}\!=\!a \sqrt{\tilde{c}} \frac{\psi_{m}}{\left\|\psi_{m}\right\|_{2}}$, then $\widetilde{\psi}_{m}\!\in \! S_{a \sqrt{\tilde{c}}}$. Since $p<4$, there exists $m_3\!>\!0$ large such that, for $m\!\geq\! max\{m_1,m_2,m_3\}$, we have
\begin{align*}
&||\Delta \tilde{\psi}_{m}||_{2}^2=\tilde{c}a^2 \frac{||\Delta \psi_{m}||_{2}^2}{\left\|\psi_{m}\right\|_{2}^2}\leq \tilde{c}a^2(3+\frac{C_6}{C_1} m^{-1})\leq 4\tilde{c}a^2,
\end{align*}
and
\begin{align*}
&\Phi_{0}(\tilde{\psi}_{m})\!+\!||\tilde{\psi}_{m}||_{2}^2\!=\!||(\Delta\!+\!1) \tilde{\psi}_{m}||_{2}^2\!-\!\frac{1}{p}{||\tilde{\psi}_{m}||}_p^p
=\frac{\tilde{c}a^2}{\left\|\psi_{m}\right\|_{2}^2}||(\Delta+1)
{\psi}_{m}||_{2}^2-\frac{1}{p}\Big(\frac{\tilde{c}a^2}
{\left\|\psi_{m}\right\|_{2}^2}\Big)^{\frac{p}{2}}{||{\psi}_{m}||}_p^p\\
&\!\leq\! \frac{\tilde{c}a^2}{\left\|\psi_{m}\right\|_{2}^2}\Big\{\frac{C_7} {m}\!-\!\frac{C_4}{p}\Big(\frac{\tilde{c}a^2}
{C_2}\Big)^{\frac{p-2}{2}}m^{-\frac{(p-2)}{2}}\Big\}\!\leq\! \frac{\tilde{c}a^2}{C_1 m^{\frac{(p-2)}{2}}}\Big\{\frac{C_7}{m^{\frac{(4-p)}{2}}} \!-\!\frac{2^{3-p\gamma_p} C_4}{p {C_2}^{\frac{p-2}{2}}} {\mu}^{(p\gamma_p-2)} a^{(p-2)} \Big\}\!<\!0.
\end{align*}
From Lemma \ref{lem3.4.0}, we deduce that $||\Delta \tilde{\psi}_{m}||_{2}\leq 2a\sqrt{\tilde{c}}<\tilde{\tau}$. Therefore, we have
$$ m_{\Phi_0}(a, \mu)=\inf_{v\in\tilde{A}_{\tilde{\tau}}} \Phi_{0}(v)\leq \Phi_{0}(\tilde{\psi}_{m})<-||\tilde{\psi}_{m}||_{2}^2=-\tilde{c} a^2.$$
By using (\ref{qan3.2.3}), we get $m(a, \mu)<-\frac{a^2\mu^2}{8}$. Since the test function $\psi$ is radial, we can similarly prove that $m_r(a,\mu)\!<\!-\frac{a^2\mu^2}{8}$.
\end{proof}


\section{Compactness of Palais-Smale sequences}
In this section, we give the compactness analysis of Palais-Smale sequences for $\bar{p}\!<\!p\!<\!4^*$.

This is a highly nontrivial issue for $\mu>0$ even in the radial space $H_{rad}^2(\R^N)$, which has been formulated in Section 1 as one of the main  difficulties.

Recall that
\begin{equation} \label{cost2.12}
{C}^{*}(N,p):=\!\frac{2^{p\gamma_p\!-\!2}p}
{2(p\gamma_p\!-\!1)C_{N,p}^p}\Big(\frac{1\!-\!\gamma_p}{\gamma_p} \Big) ^{\frac{p\gamma_p\!-\!2}{2}},~~~~~~~~{C}_{*}(N,p):=\frac{p}
{2(p\gamma_p\!-\!1)C_{N,p}^p}\Big[ \frac{2(p\!-\!2)}{p\gamma_p\!-\!1} \Big] ^{\frac{p\gamma_p\!-\!2}{2}}
\end{equation}
and
\begin{equation} \label{eq1.06}
S_{a,r}: =\Big\{ u \in H_{rad}^2({\mathbb{R}^N}): {||u||}_2^2=a^2  \Big\}.
\end{equation}
Since the
functional $E_{\mu}$ is invariant under rotation, a critical point (resp. a Palais-Smale sequence) for $\left.E_{\mu}\right|_{S_{a,r}}$ yields a real-valued radial critical point (resp. Palais-Smale sequence) for $\left.E_{\mu}\right|_{S_a}$(See Theorem 1.28 in \cite{Mwlm}). We have the following two lemmas.

\begin{lemma} \label{lem4.1}
Assume that $N\!\geq\!2$, $\overline{p}\!<\!p\!<\!4^*$, $a,\mu\!>\!0$ and $${\mu}^{p\gamma_p-2}{a}^{p-2}\!<\!\min \{ {C}^{*}(N,p), {C}_{*}(N,p), \tilde{C}(N,p)\}.$$
Let $\left\{u_{n}\right\} \subset S_{a, r}$ be a Palais-Smale
sequence for $\left.E_{\mu}\right|_{S_{a}}$ at level $c>0$ with $P_{\mu}\left(u_{n}\right) \rightarrow 0$ as $n \rightarrow \infty$. Then up to a subsequence $u_{n} \rightarrow u$ strongly in $H^{2}$, and $u \in S_{a}$ is a real-valued radial solution to (\ref{eq1.1}) for some $\lambda<-\frac{{\mu}^2}{4}$.
\end{lemma}

\begin{proof}
The proof is divided into five main steps.\\
\noindent (1) $\{u_n\}$ is bounded in $H^2$. Since $P_{\mu}(u_n)\!=\!2{||\Delta u_n||}_2^2\!-\!\mu{||\nabla u_n||}_2^2\!-\!2{\gamma_p}{||u_n||}_p^p\!=\!o(1)$, we have
$$E_{\mu}\left(u_{n}\right)=\left(\frac{1}{2}-\frac{1}{p\gamma_{p} }\right)||\Delta u_{n}||_{2}^{2}-\frac{\mu}{2}
\left(1-\frac{1}{p\gamma_{p}}\right)||\nabla u_{n}||_{2}^{2}+o(1).$$
This leads to
$
\left(\frac{1}{2}-\frac{1}{p\gamma_{p} }\right)||\Delta u_{n}||_{2}^{2}
 \leq (c+1)+\frac{\mu a}{2}
\left(1-\frac{1}{p\gamma_{p}}\right)||\Delta u_{n}||_{2}
$. We proved that $||\Delta u_{n}||_{2}$ is bounded. So $\{u_n\}$ is bounded in $H^2$ since $||u_{n}||_{2}=a$.

\noindent (2) $\exists$ Lagrange multipliers $\lambda_{n} \!\rightarrow \! \lambda \!\in\! \mathbb{R}$. Since $N\!\geq\!2$, the embedding $H_{{rad}}^{2}\left(\mathbb{R}^{N}\right) \!\hookrightarrow \! L^{r}\left(\mathbb{R}^{N}\right)$ is compact for $r \!\in\!\left(2,4^{*}\right)$, and we deduce that there exists $u\! \in\! H_{{rad}}^{2}$ such that, up to a subsequence, $u_{n} \!\rightharpoonup \!u$ weakly in $H^{2}$, $u_{n} \!\rightarrow \!u$ strongly in $L^{r}\left(\mathbb{R}^{N}\right)$ for $r \!\in\!\left(2,4^{*}\right)$, and a.e. in $\R^N$. Now, since $\left\{u_{n}\right\}$ is a Palais-Smale sequence of $\left.E_{\mu}\right|_{S_{a}}$, by the Lagrange multipliers rule there exists $\lambda_{n} \in \mathbb{R}$ such that
\begin{equation} \label{eq4.2}
\int_{\mathbb{R}^{N}} \Delta u_{n} \cdot \Delta {\varphi}- \mu \int_{\mathbb{R}^{N}}  \nabla u_{n} \cdot \nabla {\varphi}-\lambda_{n} \int_{\mathbb{R}^{N}} u_{n}{\varphi}-\int_{\mathbb{R}^{N}} \left|u_{n}\right|^{p-2} u_{n} {\varphi}=o_n(1)(|| \varphi||_{H^2})
\end{equation}
for every $\varphi \in H^{2}$, where $o_n(1) \rightarrow 0$ as $n \rightarrow \infty$. In particular, take $\varphi=u_n$, then
$$\lambda_{n} a^{2}=||\Delta u_{n}||_{2}^{2}-\mu ||\nabla u_{n}||_{2}^{2}-||u_{n}||_{p}^{p}+o_n(1).$$
and the boundedness of $\left\{u_{n}\right\}$ in $H^{2} \cap L^{p} $ implies that $\left\{\lambda_{n}\right\}$ is bounded as well; thus, up to a subsequence $\lambda_{n} \rightarrow \lambda \in \mathbb{R} .$

\noindent (3) We claim that $\lambda<0$ and $u\not \equiv 0$. Recalling that $P_{\mu}\left(u_{n}\right) \rightarrow 0$, we have
$$\lambda_{n} a^{2}=-\frac{\mu}{2}||\nabla u_{n}||_{2}^{2}+(\gamma_p-1)||u_{n}||_{p}^{p}+o_n(1).$$
Let $n\to+\infty$, then
\begin{equation} \label{eq4.001}
\lambda a^{2}=-\frac{\mu}{2} \lim_{n\to+\infty} ||\nabla u_{n}||_{2}^{2}+(\gamma_p-1)||u||_{p}^{p}.
\end{equation}
Since $\mu>0$ and $0<\gamma_{p}<1$, we deduce that $\lambda \leq 0$. If $\lambda_{n} \rightarrow 0$, we have $\mathop {\lim }\limits_{n  \to \infty}||\nabla u_{n}||_{2}^{2}=0=\mathop {\lim }\limits_{n  \to \infty} ||u_{n}||_{p}^{p}$. Using again $P_{\mu}\left(u_{n}\right) \rightarrow 0$, we have
$$ E_{\mu}\left(u_{n}\right)=-\frac{\mu}{4} ||\nabla u_{n}||_{2}^{2}+\frac{p \gamma_{p}-2} {2p}||u_{n}||_{p}^{p}+o_n(1) \rightarrow 0. $$
A contradiction with $\mathop {\lim }\limits_{n  \to \infty} E_{\mu}\left(u_{n}\right)=c>0$ and
thus $\lambda_{n} \rightarrow \lambda\!<\!0$. Next,we show that $u\not \equiv 0$. Otherwise, if $u\equiv 0$, we deduce from (\ref{eq4.001}) that $\lambda a^{2}=-\frac{\mu}{2} \mathop {\lim }\limits_{n  \to \infty}||\nabla u_{n}||_{2}^{2}$, but then we have a contradiction $0<c=\mathop {\lim }\limits_{n  \to \infty} E_{\mu}\left(u_{n}\right)=-\frac{\mu}{4} \mathop {\lim }\limits_{n  \to \infty}||\nabla u_{n}||_{2}^{2}+\frac{p \gamma_{p}-2} {2p}||u||_{p}^{p}=\frac{\lambda a^{2}}{2}<0$.

\noindent (4) Lower bound of $\mathop {\lim }\limits_{n  \to \infty}{||\Delta u_n||}_2$ and $\lambda<-\frac{{\mu}^2}{4}$.

Since $P_{\mu}(u_n)\to0$ , we deduce that
\begin{align*}
0<c&=\lim_{n\to+\infty}E_{\mu}\left(u_{n}\right)=-\frac{1}{2} \lim_{n\to+\infty} ||\Delta u_{n}||_{2}^{2}+\frac{p \gamma_{p}-1} {p}\lim_{n\to+\infty}||u_{n}||_{p}^{p} .
\end{align*}
Therefore,
\begin{equation} \label{eq4.01}
\frac{1}{2} \mathop {\lim }\limits_{n  \to \infty}||\Delta u_n||_{2}^{2}\leq \frac{p \gamma_{p}-1} {p} \mathop {\lim }\limits_{n  \to \infty}||u_n||_{p}^{p}\leq \frac{p \gamma_{p}-1} {p} C_{N,p}^p a^{p(1-\gamma_p)} \mathop {\lim }\limits_{n  \to \infty}{||\Delta u_n||}_2^{p\gamma_p}.
\end{equation}
Thus we obtain a positive lower bound of $\mathop {\lim }\limits_{n  \to \infty}{||\Delta u_n||}_2$ by
$$
 \mathop {\lim }\limits_{n  \to \infty}{||\Delta u_n||}_2\geq {\Big[ \frac{p}{2(p\gamma_p-1)C_{N,p}^pa^{p(1-\gamma_p)}} \Big]}^{\frac{1}{p\gamma_p-2}} .
$$
Case (i): If $\frac{1}{2}\!<\!\gamma_{p}\!<\!1$, then $2\gamma_{p}\!-\!1\!>\!0$. From $P_{\mu}(u_n)\!\to\!0$, we have $\mathop {\lim }\limits_{n  \to \infty}{||u_n||}_p^p\!\leq\! \frac{1}{\gamma_p}\mathop {\lim }\limits_{n  \to \infty}{||\Delta u_n||}_2^2$. Therefore
\begin{align*}
\lambda a^2&=-\mathop {\lim }\limits_{n  \to \infty}||\Delta u_n||_{2}^{2}+(2\gamma_p-1)\mathop {\lim }\limits_{n  \to \infty}||u_n||_{p}^{p} \leq \frac{\gamma_p-1}{\gamma_p} \mathop {\lim }\limits_{n  \to \infty}{||\Delta u_n||}_2^2 \\
&\leq \frac{\gamma_p-1}{\gamma_p} {\Big[ \frac{p}{2(p\gamma_p-1)C_{N,p}^pa^{p(1-\gamma_p)}} \Big]}^{\frac{2}{p\gamma_p-2}}.
\end{align*}
Since ${\mu}^{p\gamma_p-2}{a}^{p-2}<{C}^{*}(N,p)$, we deduce that $\lambda<-\frac{{\mu}^2}{4}$.  \\
Case (ii): If $0\!<\!\gamma_{p}\leq\!\frac{1}{2}$, then $2\gamma_{p}-\!1\!\leq\!0$. Inequality (\ref{eq4.01}) gives $\mathop {\lim }\limits_{n  \to \infty}{||u_n||}_p^p \!\geq\! \frac{p}{2(p\gamma_p-1)}\mathop {\lim }\limits_{n  \to \infty}{||\Delta u_n||}_2^2$.
From $P_{\mu}(u_n)\to0$, we have
\begin{align*}
\lambda a^2 &=-\mathop {\lim }\limits_{n  \to \infty}||\Delta u_n||_{2}^{2}+(2\gamma_p-1)\mathop {\lim }\limits_{n  \to \infty}||u_n||_{p}^{p}  \\
&\leq\frac{(2-p)}{2(p\gamma_p-1)} \mathop {\lim }\limits_{n  \to \infty}{||\Delta u_n||}_2^2 \leq \frac{(2-p)}{2(p\gamma_p-1)} {\Big[ \frac{p}{2(p\gamma_p-1)C_{N,p}^pa^{p(1-\gamma_p)}} \Big]}^{\frac{2}{p\gamma_p-2}}.
\end{align*}
Since ${\mu}^{p\gamma_p-2}{a}^{p-2}<{C}_{*}(N,p)$, we deduce that $\lambda<-\frac{{\mu}^2}{4}$.


\noindent (5) We claim that $u_{n} \rightarrow u$ strongly in $H^{2}$.

By the convergence of $u_{n} \rightharpoonup u\not \equiv 0$ weakly in $H^{2}$ and  (\ref{eq4.2}), we have
$$ d E_{\mu}(u) \varphi-\lambda \int_{\mathbb{R}^{N}} u {\varphi}=0,~~~~\forall \varphi \in H^2(\mathbb{R}^{N}). $$
Choosing $\varphi=u_{n}-u$ in (\ref{eq4.2}), and subtracting, we obtain
$$
\left(d E_{\mu}\left(u_{n}\right)-d E_{\mu}(u)\right)\left[u_{n}-u\right]-\lambda \int_{\mathbb{R}^{N}}\left|u_{n}-u\right|^{2}=o_n(1).
$$
Using the strong $L^p$ convergence of $\{u_n\}$, we infer that
\begin{equation} \label{eq4.02}
||\Delta( u_{n}-u)||_2^{2}-\mu||\nabla( u_{n}-u)||_2^{2}-\lambda ||u_{n}-u||_2^{2}=o_n(1).
\end{equation}
But $u_{n} \rightharpoonup u$ in $H^2(\R^N)$ dose not imply that $\nabla u_{n} \rightarrow \nabla u$ in $L^2(\R^N)$. From (\ref{eq4.02}), we have
\begin{align*}
||\Delta( u_{n}\!-\!u)||_2^{2}\!-\!\lambda ||u_{n}\!-\!u||_2^{2}\!=\!\mu||\nabla( u_{n}\!-\!u)||_2^{2}\!+\!o_n(1)
\!\leq\! \mu||\Delta( u_{n}\!-\!u)||_2|| u_{n}\!-\!u||_2\!+\!o_n(1).
\end{align*}
So we can assume that $||\Delta( u_{n}-u)||_2^{2}\geq\delta$ and $||u_{n}-u||_2^{2}\geq\delta$ for some $\delta>0$, otherwise compactness holds.
It results that
\begin{align} \label{ali4.02}
2\sqrt{-\lambda} \leq \frac{||\Delta( u_{n}-u)||_2}{||u_{n}-u||_2}-\lambda \frac{||u_{n}-u||_2}{||\Delta( u_{n}-u)||_2} \leq \mu+o_n(1)  \Longrightarrow -\frac{{\mu}^2}{4} \leq \lambda <0,
\end{align}
which contradicts with $\lambda<-\frac{{\mu}^2}{4}$. Consequently, we deduce that
$$||\Delta( u_{n}-u)||_2^{2}\to 0 ,\mbox{and}~~~~||u_{n}-u||_2^{2}\to 0 ~~~~\mbox{as}~~~~ n \to +\infty.$$
\end{proof}

\begin{lemma} \label{lem4.2}
Assume that $5\!\leq\! N\!\leq\!8$ and $\overline{p}\!<\!p\!<\!4$, or $N\!>\!8$ and $\overline{p}\!<\!p\!<\!4^*$, $a,\mu\!>\!0$ such that
${\mu}^{p\gamma_p-2}{a}^{p-2}\!<\!\tilde{C}(N,p)$. Let $\left\{u_{n}\right\} \subset S_{a, r}$ be a Palais-Smale
sequence for $\left.E_{\mu}\right|_{S_{a}}$ at level $c<-\frac{a^2\mu^2}{8}$ with $$P_{\mu}\left(u_{n}\right) \rightarrow 0~~~~\mbox{as}~~~~n\to+\infty.$$
Then up to a subsequence $u_{n} \rightarrow u$ strongly in $H^{2}$, and $u \in S_{a}$ is a real-valued radial solution to (1.1) for some $\lambda<-\frac{{\mu}^2}{4}$. We also have $\lambda \! >\! -\frac{(p\gamma_p\!-\!1)\mu^2}{4(p\gamma_p\!-\!2)} \!+\!(\gamma_{p}\!-\!1)C_{N,p}^p \Big[\frac{(p\gamma_{p}\!-\!1)}{2(p\gamma_{p}\!-\!2)}\Big]^{p\gamma_{p}} {\mu}^{p\gamma_p}a^{p-2}$.
\end{lemma}

\begin{proof}
Similar to the proof of Lemma \ref{lem4.1}, we can easily get steps (1) and (2), that is,\\
\noindent (1) $\{u_n\}$ is bounded in $H^2(\mathbb{R}^{N})$ and $u_{n} \rightharpoonup u$ weakly in $H^{2}(\mathbb{R}^{N})$. \\
\noindent (2) $\exists$ Lagrange multipliers $\lambda_{n} \rightarrow \lambda \in \mathbb{R}$. Moreover, we have
\begin{equation} \label{eq4.2.1}
\int_{\mathbb{R}^{N}} \Delta u_{n} \cdot \Delta {\varphi}- \mu \int_{\mathbb{R}^{N}}  \nabla u_{n} \cdot \nabla {\varphi}-\lambda_{n} \int_{\mathbb{R}^{N}} u_{n}{\varphi}-\int_{\mathbb{R}^{N}} \left|u_{n}\right|^{p-2} u_{n} {\varphi}=o_n(1)(|| \varphi||_{H^2})
\end{equation}
for every $\varphi \in H^{2}$, where $o_n(1) \rightarrow 0$ as $n \rightarrow \infty$. In particular, take $\varphi=u_n$, then
\begin{equation*}
\lambda_{n} a^{2}=||\Delta u_{n}||_{2}^{2}-\mu ||\nabla u_{n}||_{2}^{2}-||u_{n}||_{p}^{p}+o_n(1).
\end{equation*}

\noindent (3) We claim that $\lambda <-\frac{{\mu}^2}{4}$ and $u\not \equiv 0$.

From $P_{\mu}(u_n)=o_n(1)$ and $\lambda_n \to \lambda$, we have
\begin{equation} \label{eq4.5}
\lambda a^2=-\frac{\mu}{2} \lim_{n\to+\infty} ||\nabla u_n||_{2}^{2}+(\gamma_{p}-1) \lim_{n\to+\infty} ||u_n||_{p}^{p}
\end{equation}
and
\begin{equation} \label{eq4.6}
-\frac{\mu}{4} \mathop {\lim }\limits_{n  \to \infty}||\nabla u_{n}||_{2}^{2}+\frac{p \gamma_{p}-2} {2p} \mathop {\lim }\limits_{n  \to \infty}||u_{n}||_{p}^{p}=\lim_{n\to+\infty}E_{\mu}\left(u_{n}\right)=c
<-\frac{a^2\mu^2}{8}.
\end{equation}
Therefore, (\ref{eq4.5}) and (\ref{eq4.6}) lead to
\begin{equation*}
\lambda a^2=-\frac{\mu}{2}\lim_{n\to+\infty}||\nabla u_n||_{2}^{2}+(\gamma_{p}-1)\lim_{n\to+\infty}||u_n||_{p}^{p}
<-\frac{a^2\mu^2}{4},
\end{equation*}
which implies that $\lambda <-\frac{{\mu}^2}{4}$. Next,we show that $u\not \equiv 0$. Otherwise, if $u\equiv 0$, we have
 \begin{align*}
\lambda a^{2}=\mathop {\lim }\limits_{n  \to \infty}||\Delta u_{n}||_{2}^{2}-\mu \mathop {\lim }\limits_{n  \to \infty}||\nabla u_{n}||_{2}^{2}-||u||_{p}^{p} \geq \mathop {\lim }\limits_{n  \to \infty}||\Delta u_{n}||_{2}^{2}-\mu a \mathop {\lim }\limits_{n  \to \infty}||\Delta u_{n}||_{2},
\end{align*}
but then we have a contradiction
\begin{align*}
\mu\geq \frac{\mathop {\lim }\limits_{n  \to \infty}||\Delta u_{n}||_{2}}{a}-\lambda \frac{a}{\mathop {\lim }\limits_{n  \to \infty}||\Delta u_{n}||_{2}}\geq 2 \sqrt{-\lambda}>\mu.
\end{align*}

\noindent (4) We claim that $u_{n} \rightarrow u$ strongly in $H^{2}$.


If $u_{n} \not \rightarrow u$ strongly in $H^{2}$, we can proceed as in (\ref{ali4.02}) and get $-\frac{{\mu}^2}{4} \leq \lambda <0$. Therefore, it must be $u_{n} \rightarrow u$ strongly in $H^{2}$. Finally, we also deduce from $P_{\mu}(u_n)=o_n(1)$ that
 \begin{align} \label{li4.11}
\left(\frac{1}{2}\!-\!\frac{1}{p\gamma_{p} }\right) \mathop {\lim }\limits_{n  \to \infty} ||\Delta u_{n}||_{2}^{2}\!-\!\frac{\mu}{2}
\left(1\!-\!\frac{1}{p\gamma_{p}}\right) \mathop {\lim }\limits_{n  \to \infty}||\nabla u_{n}||_{2}^{2}\!=\!\lim_{n\to+\infty}E_{\mu}\left(u_{n}\right)\!=\!c
\!<\!-\frac{a^2\mu^2}{8}.
\end{align}
Then, \eqref{eq4.6} and \eqref{li4.11} imply that
\begin{align} \label{li4.12}
\frac{\mu a}{2} < \mathop {\lim }\limits_{n  \to \infty} ||\Delta u_{n}||_{2} < \frac{(p\gamma_{p}-1)\mu a}{2(p\gamma_{p}-2)}.
\end{align}
By using (\ref{eq4.5}), the lower bound of $\lambda$ follows directly from the fact that
\begin{align*}
\lambda a^2
&\geq -\frac{\mu a}{2} \lim_{n\to+\infty} ||\Delta u_n||_{2}+(\gamma_{p}-1) C_{N,p}^p a^{p(1-\gamma_p)} \lim_{n\to+\infty}{||\Delta u_n||}_2^{p\gamma_p}\\
&> -\frac{(p\gamma_p-1)\mu^2 a^2}{4(p\gamma_p-2)} +(\gamma_{p}-1)C_{N,p}^p \Big[\frac{(p\gamma_{p}-1)}{2(p\gamma_{p}-2)}\Big]^{p\gamma_{p}} {\mu}^{p\gamma_p}a^{p} .
\end{align*}
\end{proof}

\section{Proof of Theorems \ref{th1.1}-\ref{th1.3}}

In this section, we first prove the existence results, i.e. Theorem \ref{th1.1}-(1),(2),(3). Next, we prove the asymptotic properties of $m_r(a,\mu)$ and $\sigma(a,\mu)$ as $\mu \!\to\! 0^{+}$, i.e. Theorem \ref{th1.1}-(4),(5). Finally, we prove the asymptotic properties of $m_r(a,\mu)$ as $a \!\to\! 0^{+}$, i.e. Theorem \ref{th1.3}.

The proof of existence results in Theorem \ref{th1.1} is divided into two parts. Firstly, we prove the existence of a local minimizer for $\left.E_{\mu}\right|_{S_{a}}$. Secondly, we construct a mountain pass type critical point for $\left.E_{\mu}\right|_{S_{a}}$. The later relies on a refined version of the min-max principle by N. Ghoussoub \cite{GN}, the forth coming Lemma \ref{lem3.10}, and was already applied in \cite{NSoa,nsoa}.

\begin{definition}\label{def3.9}
Let $B$ be a closed subset of $X$. We shall say that a class $\mathcal{F}$ of compact subsets of $X$ is a homotopy-stable family with extended boundary $B$ if for any set $A$ in $\mathcal{F}$ and any $\eta\in C([0,1]\times X;X)$ satisfying $\eta(t,x)=x$ for all $(t,x)\in (\{0\}\times X)\cup ([0,1]\times B)$ we have that $\eta(\{1\}\times A)\in \mathcal{F}$.
\end{definition}

\begin{lemma}(\cite{GN}, Theorem 5.2)\label{lem3.10}
Let $\varphi$ be a $C^{1}$-functional on a complete connected $C^{1}$-Finsler manifold $X$ and consider a homotopy-stable family $\mathcal{F}$ with an extended closed boundary $B$. Set $c=c(\varphi,\mathcal{F})$ and let $F$ be a closed subset of $X$ satisfying \\
$(1)$~~~~~~~~$(A \cap F)\backslash B \neq \emptyset \quad \text { for every } A \in \mathcal{F}$ \\
$(2)$~~~~~~~~$\sup \varphi(B) \leq c \leq \inf \varphi(F)$.  \\
Then, for any sequence of sets $(A_{n})_{n}$ in $\mathcal{F}$ such that $\lim_{n}\sup_{A_{n}}\varphi=c$, there exists a sequence $(x_{n})_{n}$ in $X$ such that
$$\lim_{n \rightarrow +\infty}\varphi(x_{n})=c,\ \ \lim_{n \rightarrow +\infty}\|d\varphi(x_{n})\|=0,\ \ \lim_{n \rightarrow +\infty}dist(x_{n},F)=0,\ \ \lim_{n \rightarrow +\infty}dist(x_{n},A_{n})=0.$$
\end{lemma}

We also need the following result, where $T_uS_a$ denotes the tangent space to $S_a$ in $u$.
\begin{lemma} (\cite{NSoa}, Lemma 5.8) \label{lam5.3}
For $u \in S_a$ and $s\in \mathbb{R}$ the map
$$
T_{u}S_a \rightarrow T_{s \star u}S_a, \quad \varphi \mapsto s \star \varphi
$$
is a linear isomorphism with inverse $\psi \mapsto(-s) \star \psi$.  \\
\end{lemma}

\noindent \textbf{Proof of Theorem \ref{th1.1}-(1),(2),(3):}

We choose $S_{a,r}=S_a \cap H_{{rad}}^{2} $ as our work space. Since the
functional $E_{\mu}$ is invariant under rotation, a critical point (resp. a Palais-Smale sequence) for $\left.E_{\mu}\right|_{S_{a,r}}$ yields a real-valued radial critical point (resp. Palais-Smale sequence) for $\left.E_{\mu}\right|_{S_a}$(See Theorem 1.28 in \cite{Mwlm}).

\noindent \textbf{(i) Existence of a local minimizer.}

Let us consider a minimizing sequence $\{v_n\}$ for $\left.E_{\mu}\right|_{A^r_{R_{0}}}$. For every $n$ we can take $s_{v_{n}} \star v_{n} \in S_{a,r} \cap \mathcal{P}_{+}^{a,\mu}$, observing that then
by Lemma \ref{lem3.4} and Corollary \ref{coro3.5} $||\Delta\left(s_{v_{n}} \star v_{n}\right)||_{2}<R_{0}$ and
$$E_{\mu}\left(s_{v_{n}} \star v_{n}\right)=\min \left\{E_{\mu}(s \star v_{n}) : s \in \mathbb{R} \text { and }||\Delta(s \star v_{n})||_{2}<R_{0}\right\} \leq E_{\mu}\left(v_{n}\right);$$
in this way we obtain a new minimizing sequence $\left\{w_{n}=s_{v_{n}} \star v_{n}\right\}$, with
$$w_{n} \in S_{a,r} \cap \mathcal{P}_{+}^{a,\mu}~~~~\text{and}~~~~P_{\mu}(w_{n})=0$$
for every $n$. By Lemma \ref{lem3.6}, $||\Delta w_{n}||_2<R_0-\rho$ for every $n\in\mathbb{N}$ and some $\rho>0$ sufficiently small, and hence the Ekeland's variational principle yields in a standard way the existence of a new minimizing sequence $\left\{u_{n}\right\} \subset A^r_{R_{0}}$ for $m_r(a, \mu)<0$, with the property that $||u_{n}-w_{n}||_{H^2} \rightarrow 0$ as $n \rightarrow +\infty$, which is also a Palais-Smale sequence for $E_{\mu}$ on $S_{a,r}$. The condition $||u_{n}-w_{n}||_{H^2} \rightarrow 0$ implies
$$||\Delta u_{n}||_2 < R_0-\rho~~~~\text{and}~~~~P_{\mu}(u_{n}) \rightarrow 0~~~~\mbox{as}~~~~n \rightarrow \infty.$$
If $p\!<\!4$, we deduce from Lemma \ref{lema3.4.1} that $m_r(a, \mu)\!<\!-\frac{a^2\mu^2}{8}$, and hence $\{u_n\}$ satisfies all the assumptions of Lemma \ref{lem4.2}. Consequently, up to a subsequence $u_{n} \rightarrow \tilde{u}_{\mu}$ strongly in $H^2$, $\tilde{u}_{\mu}$ is an interior local minimizer for $\left.E_{\mu}\right|_{A^r_{R_{0}}}$, and solves (\ref{eq1.1}) for some $\tilde{\lambda}<-\frac{\mu^2}{4}$. We also have $\tilde{\lambda} \! >\! -\frac{(p\gamma_p\!-\!1)\mu^2}{4(p\gamma_p\!-\!2)} \!+\!(\gamma_{p}\!-\!1)C_{N,p}^p \big[\frac{(p\gamma_{p}\!-\!1)}{2(p\gamma_{p}\!-\!2)}\big]^{p\gamma_{p}} {\mu}^{p\gamma_p}a^{p-2}$ by Lemma \ref{lem4.2}.

Since any critical point of
$E_{\mu}|_{S_{a,r}}$ lies in $\mathcal{P}_{a,\mu}\cap S_{a,r}$ and $m_r(a, \mu)=\inf _{\mathcal{P}_{a,\mu}\cap S_{a,r}} E_{\mu}$ ( see Lemma \ref{lem3.6}), we see that $\tilde{u}_{\mu}$ is a radial ground state for $\left.E_{\mu}\right|_{S_{a,r}}$. It only remains to prove that any radial ground state of $E_{\mu}|_{S_{a,r}}$ is a local minimizer of $E_{\mu}$ in $A^r_{R_0}$. Let then $u$ be a radial critical point of $E_{\mu}|_{S_{a,r}}$ with $E_{\mu}(u)=m_r(a, \mu)=\inf_{\mathcal{P}_{a,\mu} \cap S_{a,r} } E_{\mu}$. Since $E_{\mu}(u)<0<\inf_{\mathcal{P}^{a,\mu}_{-}\cap S_{a,r}} E_{\mu}$, necessarily $u \in \mathcal{P}^{a,\mu}_{+}\cap S_{a,r}$. Then Corollary \ref{coro3.5} implies that $\mathcal{P}^{a,\mu}_{+}\cap S_{a,r}\subset A^r_{R_0}$, it results that $||\Delta u||_2<R_0$, and as a consequence $u$ is a local minimizer for $E_{\mu}|_{A^r_{R_0}}$.

\noindent \textbf{(ii) Existence of a Mountain pass type solution.}

We focus now on the existence of a second critical point for $\left.E_{\mu}\right|_{S_a}$. Denote $E_{\mu}^{c}=\{u \in S_a : E_{\mu}(u) \leq c\}$. Motivated by \cite{LjEa}, we define the augmented functional $\tilde{E_{\mu}}: \mathbb{R} \times  H^{2} \rightarrow \mathbb{R}$
$$\tilde{E_{\mu}}(s,u):=E_{\mu}(s \star u)=E_{\mu}(s \star u)=\frac{e^{4s}}{2} {||\Delta u||}_2^2-\frac{\mu}{2}e^{2s} {||\nabla u||}_2^2-\frac{e^{2p\gamma_{p} s}}{p}{||u||}_p^p$$
and consider the restriction $\tilde{E_{\mu}}|_{\mathbb{R}\times {S_a} }$. Notice that $S_{a,r}=H_{{rad}}^{2} \cap S_a$ and $\tilde{E_{\mu}}$ is of class $C^1$. Moreover, 
a Palais-Smale sequence for $\tilde{E_{\mu}}|_{\mathbb{R}\times {S_{a,r}} }$ is a Palais-Smale sequence for $\tilde{E_{\mu}}|_{\mathbb{R}\times {S_a} }$(See Theorem 1.28 in \cite{Mwlm}).

We introduce the minimax class
$$
\Gamma :=\left\{\gamma(\tau)=\big(\zeta(\tau),\beta(\tau)\big) \in C\left([0,1], \mathbb{R}\times S_{a,r}\right) ; \gamma(0) \in (0,\mathcal{P}_{+}^{a,\mu}), \gamma(1) \in (0, E_{\mu}^{2m_r(a, \mu)})\right\}.
$$
The family $\Gamma$ is not empty. Indeed, for every $u \in S_{a,r}$, by Lemma \ref{lem3.4} we know that there exists $s_{1} \gg 1$ such that
\begin{equation} \label{uqe6.1}
\gamma_{u} : \tau \in[0,1] \mapsto   \big(0,\left((1-\tau) s_{u}+\tau s_{1}\right) \star u \big ) \in \mathbb{R}\times S_{a,r}
\end{equation}
is a path in $\Gamma$ (recall that $s \in \mathbb{R} \mapsto s \star u \in S_{a,r}$ is continuous, $s_{u} \star u \in \mathcal{P}_{+}^{a,\mu}$ and  $E_{\mu}(s \star u) \rightarrow-\infty$ as $s \rightarrow +\infty$). Thus, the minimax value
$$
\sigma(a, \mu) :=\inf _{\gamma \in \Gamma} \max _{(s,u) \in \gamma([0,1])} \tilde{E}_{\mu}(s,u)
$$
is a real number. We claim that
\begin{equation} \label{uqe6.2}
\forall \gamma \in \Gamma~~~~\text{there exists}~~~~\tau_{\gamma} \in (0,1)~~~~\text{such that}~~~~\zeta(\tau_{\gamma}) \star \beta(\tau_{\gamma}) \in \mathcal{P}_{-}^{a,\mu}.
\end{equation}
Indeed, since $\gamma(0)=\big(\zeta(0),\beta(0)\big) \in (0,\mathcal{P}_{+}^{a,\mu})$, by Corollary \ref{coroll2.1} and Lemma \ref{lem3.4}, we have $t_{\zeta(0)\star \beta(0)}\!=\!t_{\beta(0)}\!>\!s_{\beta(0)}\!=\!0$; since $E_{\mu}(\beta(1))\!=\!\tilde{E}_{\mu}(\gamma(1)) \!\leq \!2 m_r(a, \mu)$, by Lemma \ref{lem3.7}, we have
$$t_{\zeta(1)\star \beta(1)}=t_{\beta(1)}<0,$$
and moreover the map $t_{\zeta(\tau)\star \beta(\tau)}$ is continuous in $\tau$(we refer again to Lemma \ref{lem3.4} and recall that $s \in \mathbb{R} \mapsto s \star u \in S_{a,r}$ is continuous). It follows that for every $\gamma \in \Gamma$ there exists $\tau_{\gamma} \in(0,1)$ such that $t_{\zeta(\tau_{\gamma})\star \beta(\tau_{\gamma})}=0$, and so $\zeta(\tau_{\gamma}) \star \beta(\tau_{\gamma}) \in \mathcal{P}_{-}^{a,\mu}$. Thus (\ref{uqe6.2}) holds.

For every $\gamma \in \Gamma$, by (\ref{uqe6.2}) we have
\begin{equation} \label{uqe6.3}
\max _{\gamma([0,1])} \tilde{E}_{\mu} \geq \tilde{E}_{\mu} \left(\gamma\left(\tau_{\gamma}\right)\right)={E}_{\mu}(\zeta(\tau_{\gamma}) \star \beta(\tau_{\gamma})) \geq \inf _{\mathcal{P}^{a,\mu}_{-\cap S_{a,r}}} E_{\mu},
\end{equation}
which gives $\sigma(a, \mu) \geq \inf _{\mathcal{P}_{-}^{a,\mu} \cap S_{a,r}} E_{\mu}$. On the other hand, if $u \in \mathcal{P}^{a,\mu}_{-} \cap S_{a,r},$ then $\gamma_{u}$ defined in (\ref{uqe6.1}) is a path in $\Gamma$ with
$$
E_{\mu}(u)=\tilde{E}_{\mu}(0,u)=\max _{\gamma_u([0,1])} \tilde{E}_{\mu} \geq \sigma(a, \mu),
$$
which gives $ \inf _{\mathcal{P}^{a,\mu}_{-} \cap S_{a,r}} E_{\mu} \geq \sigma(a, \mu)$. This, Corollary \ref{coro3.5} and Lemma \ref{lem3.8} imply that
\begin{equation}  \label{uqe6.4}
\sigma(a, \mu)\!=\!\inf _{\mathcal{P}^{a,\mu}_{-} \cap S_{a,r}} E_{\mu}\!>\!0 \!\geq\! \sup _{\left(\mathcal{P}^{a,\mu}_{+} \cup E_{\mu}^{2 m_r(a, \mu)}\right) \cap S_{a,r}} E_{\mu}\!=\!\sup _{\left((0,\mathcal{P}^{a,\mu}_{+}) \cup (0,E_{\mu}^{2 m_r(a, \mu)})\right) \cap  (\mathbb{R}\times S_{a,r})}  \tilde{E}_{\mu}.
\end{equation}

Let $\gamma_n(\tau)=\big(\zeta_n(\tau),\beta_n(\tau)\big)$ be any minimizing sequence for $\sigma(a, \mu)$ with the property that $\zeta_n(\tau)\equiv 0$ for every $\tau \in[0, 1]$ (Notice that, if $\{\gamma_n=\big(\zeta_n,\beta_n\big)\}$ is a minimizing sequence, then also $\{(0,\zeta_n\star {\beta_n})\}$ has the same property). Take $$X=\mathbb{R}\times S_{a,r},~~~~\mathcal{F}=\{\gamma([0,1]):~~\gamma \in \Gamma\},~~~~B=(0,\mathcal{P}_{+}^{a,\mu}) \cup (0, E_{\mu}^{2m_r(a, \mu)}),$$
$$F=\{(s,u)\in\mathbb{R}\times S_{a,r} ~~~~|~~~~\tilde{E}_{\mu}(s,u)\geq \sigma(a, \mu) \},~~~~A=\gamma([0,1]),~~~~A_n=\gamma_n([0,1])$$
in Lemma \ref{lem3.10}. We need to checked that $\mathcal{F}$ is a homotopy stable family of compact subsets of $X$ with extended closed boundary $B$, and that $F$ is a dual set for $\mathcal{F}$, in the sense that assumptions (1) and (2) in Lemma \ref{lem3.10} are satisfied.

Indeed, since $\sigma(a, \mu)=\inf _{\mathcal{P}^{a,\mu}_{-} \cap S_{a,r}} E_{\mu}$, (\ref{uqe6.3}) $\Rightarrow \gamma\left(\tau_{\gamma}\right)=(\zeta(\tau_{\gamma}), \beta(\tau_{\gamma}))\in A \cap F$, (\ref{uqe6.4}) $\Rightarrow F \cap B=\emptyset$ and (2) in Lemma \ref{lem3.10}, then $A \cap F\not=\emptyset$ and $F \cap B=\emptyset$ give (1) in Lemma \ref{lem3.10}. For every $ \gamma \in \Gamma$, since $\gamma(0) \in (0,\mathcal{P}_{+}^{a,\mu})$ and $\gamma(1) \in (0, E_{\mu}^{2m_r(a, \mu)})$, we have $\gamma(0), \gamma(1) \in B$. Then for any set $A$ in $\mathcal{F}$ and any $\eta\in C([0,1]\times X;X)$
satisfying $\eta(t,x)=x$ for all $(t,x)\in (\{0\}\times X)\cup ([0,1]\times B)$, it holds that $\eta(1,\gamma(0))=\gamma(0),~~~~\eta(1,\gamma(1))=\gamma(1)$. So we have $\eta(\{1\}\times A)\in \mathcal{F}$.

Consequently, by Lemma \ref{lem3.10}, there exists a Palais-Smale sequence $\{(s_n,w_n)\}\subset \mathbb{R}\times S_{a,r}$ for $\tilde{E_{\mu}}|_{\mathbb{R}\times {S_{a,r}} }$ at level $\sigma(a, \mu)>0$ such that
\begin{equation}  \label{uqe6.5}
\partial_{s} \tilde{E_{\mu}}\left(s_{n}, w_{n}\right) \rightarrow 0 \quad \text { and } \quad\left\|\partial_{u} \tilde{E_{\mu}}\left(s_{n}, w_{n}\right)\right\|_{\left(T_{w_{n}} S_{a,r}\right)^{*}} \rightarrow 0 \quad \text { as } n \rightarrow \infty,
\end{equation}
with the additional property that
\begin{equation}   \label{uqe6.6}
\left|s_{n}\right|+\operatorname{dist}_{H^{2}}\left(w_{n}, \beta_{n}([0,1])\right) \rightarrow 0 \quad \text { as } n \rightarrow \infty.
\end{equation}
The first condition in (\ref{uqe6.5}) reads $P_{\mu}\left(s_n \star w_{n}\right) \rightarrow 0$, and the second condition in (\ref{uqe6.5}) gives
\begin{equation} \label{uqe6.7}
\begin{gathered}
e^{4s_n} \int_{\mathbb{R}^{N}} \Delta w_{n} \cdot \Delta {\varphi}\!-\! \mu e^{2s_n} \int_{\mathbb{R}^{N}}  \nabla w_{n} \cdot \nabla {\varphi}\!-\!e^{2p\gamma_{p} s_n}\int_{\mathbb{R}^{N}} \left|w_{n}\right|^{p-2} w_{n} {\varphi}\!=\!o_n(1)(|| \varphi||_{H^2})
\end{gathered}
\end{equation}
for any $\varphi \in T_{w_{n}} S_{a,r}$. Since $\left|s_{n}\right|$ is bounded , due to (\ref{uqe6.6}), we have
\begin{equation} \label{uqe6.8}
d E_{\mu}\left(s_{n} \star w_{n}\right)\left[s_{n} \star \varphi\right]=o_n(1)\|\varphi\|_{H^2}=o_n(1)\left\|s_{n} \star \varphi\right\|_{H^2}~~\text{as}~~n \rightarrow \infty, \forall \varphi \in T_{w_{n}} S_{a,r}.
\end{equation}
By Lemma \ref{lam5.3}, (\ref{uqe6.8}) implies that $\{u_n:=s_n\star w_n\} \subset S_{a,r}$ is a Palais-Smale sequence for $E_{\mu}|_{S_{a,r}}$ (thus a Palais-Smale sequence for $E_{\mu}|_{S_{a}}$, since the problem is invariant under rotations) at level $\sigma(a,\mu)>0$, with $P_{\mu}(u_n)\to0$. Therefore, all the assumptions of Lemma \ref{lem4.1} are satisfied, and we deduce that up to a subsequence $u_{n} \rightarrow \hat{u}_{\mu}$ strongly in $H^2$, with $\hat{u}_{\mu} \in S_{a,r}$ real-valued radial solution to $(1.1)$ for some $\hat{\lambda}<-\frac{\mu^2}{4}$.

Suppose in addition that $N\!<\!8$ and $p\!<\!\min \{\frac{2(N-2)}{N-4},4\}$, then we deduce from $\tilde{\lambda}, \hat{\lambda}<-\frac{{\mu}^2}{4}$ that $\tilde{u}_{\mu}$ and $\hat{u}_{\mu}$ are sign-changing by Theorem 3.7 in \cite{DBon}, which is also adopted in \cite{dbJB} to obtain radial sign-changing normalized solutions to \eqref{eq1.1} with $\mu < 0$. \qed \\

\begin{remark} \label{mak1.6}
As a byproduct of Theorem \ref{th1.1}, we obtain the following decay results by Theorem 3.10 in \cite{DBon}. Let ${\lambda}<-\frac{{\mu}^2}{4}$ and $u$ be a classical solution to
$$   {\Delta}^{2}u+\mu \Delta u-{\lambda}u={|u|}^{p-2}u~~~~\text{in}~~~~\R^{N}~~~~\text{and}~~~~\lim_{|x| \to\infty}u(x)=0,  $$
then we have $|u(x)|\leq \frac{C}{\sqrt{-\mu^2-4\lambda}}
e^{-(\frac{\sqrt{2\sqrt{-\lambda}+\mu}}{2}-\varepsilon)|x|}$
for any $\varepsilon>0$ and $|x|$ large enough, where $C>0$ is a constant.\\
\end{remark}

To obtain the asymptotic property of $\sigma(a, \mu)$ as $\mu \to 0^{+}$, we need to study equation (\ref{eq1.1}) with $\mu=0$. We consider once again the Pohozaev manifold $\mathcal{P}_{a,0}$, defined in Section 2 and the decomposition $\mathcal{P}_{a,0}=\mathcal{P}_{+}^{a,0} \cup \mathcal{P}_{0}^{a,0} \cup \mathcal{P}_{-}^{a,0}$. Following the argument in the previous Section 2 or section 6 of \cite{NSoa} , we can prove the following Lemmas \ref{lem6.1}-\ref{lem6.4}.
\begin{lemma}\label{lem6.1}
Let $N\geq2$, $\overline{p}<p<4^*$ and $a>0$. Then $\mathcal{P}_{0}^{a, 0}=\emptyset$, and $\mathcal{P}_{a,0}$ is a smooth manifold of codimension 2 in $H^{2}(\R^N)$.
\end{lemma}

\begin{lemma}\label{lem6.2}
Let $N\geq2$, $\overline{p}<p<4^*$ and $a>0$. For every $u \in S_{a}$, there exists a unique $t_{u} \in \mathbb{R}$ such that $t_{u} \star u \in \mathcal{P}_{a,0}$. $t_{u}$ is the unique critical point of the function $\Psi_{u}^{0}$, and is a strict maximum point at positive level. Moreover: \\
\noindent $(1)$ $\mathcal{P}_{a,0}=\mathcal{P}_{-}^{a,0}$.\\
\noindent $(2)$ $\Psi_{u}^{0}$ is strictly decreasing and concave on $\left(t_{u},+\infty\right)$.\\
$(3)$ The maps $u \in S_{a} \mapsto t_{u} \in \mathbb{R}$ are of class $C^1$.\\
$(4)$ If $P_{0}(u)<0$, then $t_{u}<0$.
\end{lemma}

\begin{lemma}\label{lem6.3}
Let $N\!\geq\!2$, $\overline{p}\!<\!p\!<\!4^*$ and $a\!>\!0$. It results that $m(a, 0) :=\!\inf _{u \in \mathcal{P}_{a,0}} E_{0}(u)\!>\!0$.
\end{lemma}

\begin{lemma}\label{lem6.4}
Let $N\!\geq\!2$, $\overline{p}\!<\!p\!<\!4^*$ and $a\!>\!0$. There exists $k\!>\!0$ small such that
$$
0<\sup_{\overline{A_{k}}} E_{0}<m(a, 0) \quad \text{and} \quad u\in \overline{A_{k}} \Longrightarrow E_{0}(u)>0,~~~~P_{0}(u)>0,
$$
where $A_{k} :=\left\{u \in S_a :||\Delta u||_{2} < k\right\}$.
\end{lemma}

\begin{lemma} \label{lemma6.5}
Let $N\geq2$, $\overline{p}<p<4^*$ and $a>0$. Then, there exists a real valued radial mountain pass type critical point $u_0$ for $E_0|_{S_a}$ at a positive level
$$\sigma(a,0)=\inf _{{\mathcal{P}_{-}^{a,0}} \cap S_{a,r}} E_0=\inf _{{\mathcal{P}_{a,0}} \cap S_{a,r}} E_0=E_0(u_0).$$
That is, ${u}_{0}$ is a radial ground state to (\ref{eq1.1}) obtained for $\mu=0$.
\end{lemma}

\begin{proof}
Using the arguments in the proof of Theorem \ref{th1.1}-(2) or section 6 in \cite{NSoa} and Lemmas \ref{lem6.1}-\ref{lem6.4}, we can drive the desired results.
\end{proof}

\begin{lemma}\label{lem6.7}
For any $\mu\!>\!0$ satisfying ${\mu}^{p\gamma_p-2}{a}^{p-2}\!<\!\min \{ {C}^{*}(N,p), {C}_{*}(N,p), \tilde{C}(N,p)\}$, we have
$$
\sigma(a, \mu)=\inf _{u \in S_{a, r}} \max _{s \in \mathbb{R}} E_{\mu}(s \star u), \quad \text { and } \quad \sigma(a, 0)=\inf _{u \in S_{a, r}} \max _{s \in \mathbb{R}} E_{0}(s \star u).
$$
\end{lemma}

\begin{proof}
From (\ref{uqe6.4}), we have $\sigma(a, \mu)=\inf _{{\mathcal{P}_{-}^{a, \mu}} \cap S_{a, r}} E_{\mu}=E_{\mu}\left(\hat{u}_{\mu}\right)$. Then, by Lemma \ref{lem3.4},
$$
\sigma(a, \mu)=E_{\mu}\left(\hat{u}_{\mu}\right)=\max _{s \in \mathbb{R}} E_{\mu}\left(s \star \hat{u}_{\mu}\right) \geq \inf _{u \in S_{a,r}} \max _{s \in \mathbb{R}} E_{\mu}(s \star u).
$$
On the other hand, for any $u \in S_{a, r}$ we have $t_{u, \mu} \star u \in \mathcal{P}_{-}^{a, \mu}$, and hence
$$
\max _{s \in \mathbb{R}} E_{\mu}(s \star u)=E_{\mu}\left(t_{u, \mu} \star u\right) \geq \sigma(a, \mu).
$$
By Lemma \ref{lem6.2} and Lemma \ref{lemma6.5}, we can similarly prove the case of $\sigma(a, 0)$. 
\end{proof}

\begin{lemma}\label{lem6.8}
For any $0\leq\mu_{1}<\mu_{2}$, with ${\mu_{2}}$ satisfying $$\mu_{2}^{p\gamma_p-2}{a}^{p-2}<\min \{ {C}^{*}(N,p), {C}_{*}(N,p), \tilde{C}(N,p)\},$$
it results that $\sigma\left(a, \mu_{2}\right) \leq \sigma\left(a, \mu_{1}\right) \leq \sigma(a, 0)$.
\end{lemma}

\begin{proof}
We only prove the case of $0<\mu_{1}<\mu_{2}$. By Lemma \ref{lem6.7}
$$
\sigma\left(a, \mu_{2}\right) \leq \max _{s \in \mathbb{R}} E_{\mu_{2}}\left(s \star \hat{u}_{\mu_{1}}\right) \leq \max _{s \in \mathbb{R}} E_{\mu_{1}}\left(s \star \hat{u}_{\mu_{1}}\right)=E_{\mu_{1}}\left(\hat{u}_{\mu_{1}}\right)=\sigma\left(a, \mu_{1}\right)
$$
and
$$
\sigma\left(a, \mu_{1}\right) \leq \max _{s \in \mathbb{R}} E_{\mu_{1}}\left(s \star {u}_{0}\right) \leq \max _{s \in \mathbb{R}} E_{0}\left(s \star {u}_{0}\right)=E_{0}\left({u}_{0}\right)=\sigma(a,0).
$$\\
\end{proof}

\noindent \textbf{Proof of Theorem \ref{th1.1}-(4): convergence of $\tilde{u}_{\mu}$.}

For $a>0$ fixed, we deduce from Lemma \ref{lem3.1} that $R_{0}(a, \mu) \rightarrow 0$ as $\mu \rightarrow 0^{+}$, and hence $||\Delta \tilde{u}_{\mu}||_{2}<R_{0}(a, \mu) \rightarrow 0$ as well. Moreover
$$
0>m_r(a, \mu)=E_{\mu}\left(\tilde{u}_{\mu}\right) \geq \frac{1}{2}{||\Delta \tilde{u}_{\mu} ||}_2^2-\frac{\mu a}{2}{||\Delta \tilde{u}_{\mu} ||}_2-\frac{C_{N,p}^p}{p}a^{p(1-\gamma_p)}{||\Delta \tilde{u}_{\mu}||}_2^{p\gamma_p} \rightarrow 0,
$$
which implies that $m_r(a, \mu)\to 0$. From Lemma \ref{lem4.2}, we have $\tilde{\lambda} \to 0^-$ as $\mu \rightarrow 0^{+}$. \qed   \\

\noindent \textbf{Proof of Theorem \ref{th1.1}-(5): convergence of ${\hat{u}_{\mu}}$.}

Consider $\left\{\hat{u}_{\mu} : 0<\mu<\overline{\mu}\right\}$ with $\overline{\mu}$ small enough. Since $\hat{u}_{\mu} \in \mathcal{P}_{a, \mu}$, from Lemma \ref{lem6.8}, we have
\begin{equation}
\begin{aligned}
 \sigma(a,0) & \geq \sigma\left(a, \mu\right) =E_{\mu}\left({\hat{u}_{\mu}} \right)=\left(\frac{1}{2}-\frac{1}{p\gamma_{p} }\right)||\Delta {\hat{u}_{\mu}}||_{2}^{2}-\frac{\mu}{2}
\left(1-\frac{1}{p\gamma_{p}}\right)||\nabla {\hat{u}_{\mu}}||_{2}^{2} \\
& \geq \left(\frac{1}{2}-\frac{1}{p\gamma_{p} }\right)||\Delta {\hat{u}_{\mu}}||_{2}^{2}-\frac{\mu a}{2}
\left(1-\frac{1}{p\gamma_{p}}\right)||\nabla {\hat{u}_{\mu}}||_{2}.
\end{aligned}
\end{equation}
This implies that $\left\{\hat{u}_{\mu}\right\}$ is bounded in $H^2$. Since each $\hat{u}_{\mu}$ is a real-valued radial function in $S_a$, we deduce that up to a subsequence $\hat{u}_{\mu} \rightharpoonup \hat{u}$ weakly in $H^{2}$,  strongly in $L^{r}$ for $2<r<4^*$ and a.e. in $\mathbb{R}^{N}$, as $\mu \rightarrow 0^{+}$. Using the fact that $\hat{u}_{\mu}$ solves
\begin{equation}\label{eq6.2}
{\Delta}^{2}\hat{u}_{\mu}+\mu \Delta \hat{u}_{\mu}-\hat{\lambda}_{\mu}\hat{u}_{\mu}
={|\hat{u}_{\mu}|}^{p-2}\hat{u}_{\mu}~~~~\text{in}~~~~ \R^{N}
\end{equation}
for $\hat{\lambda}_{\mu}<-\frac{\mu^2}{4}$ and $P_{\mu}\left(\hat{u}_{\mu}\right)=0$, we infer that $\hat{\lambda}_{\mu} a^{2}=-\frac{\mu}{2}||\nabla {\hat{u}_{\mu}}||_{2}^{2}+(\gamma_p-1)||{\hat{u}_{\mu}}||_{p}^{p}$ and hence also $\hat{\lambda}_{\mu}$ converges (up to a subsequence) to some $\hat{\lambda} \leq 0$, with
$$\hat{\lambda} a^{2}=(\gamma_p-1)||{\hat{u}}||_{p}^{p} .$$
Therefore, we have $\hat{\lambda}=0$ if only if the weak limit $\hat{u} \equiv 0$. We claim that $\hat{\lambda}<0$. In fact, $\hat{u}_{\mu} \rightharpoonup \hat{u}$ weakly in $H^{2}$ implies that $\hat{u}$ is a weak radial (and real) solution to
\begin{equation}\label{eq6.3}
{\Delta}^{2}\hat{u}-\hat{{\lambda}}\hat{u}={|\hat{u}|}^{p-2}\hat{u} ~~~~\text{in}~~~~{\R}^N,
\end{equation}
and in particular by the Pohozaev identity $||\Delta {\hat{u}}||_{2}^{2}=\gamma_{p}||\nabla {\hat{u}}||_{2}^{2}$. But then, using the boundedness of $\left\{\hat{u}_{\mu}\right\}$ and Lemma \ref{lem6.8}, we deduce that
\begin{align*}
E_{0}(\hat{u}) &=\frac{p\gamma_{p}-2}{2p}||\hat{u}||_{p}^{p}=\lim _{\mu \rightarrow 0^{+}}\left[\frac{p\gamma_{p}-2}{2p} ||\hat{u}_{\mu}||_{p}^{p}-\frac{\mu}{4}||\nabla \hat{u}_{\mu}||_{2}^{2}\right] \\ &=\lim _{\mu \rightarrow 0^{+}} E_{\mu}\left(\hat{u}_{\mu}\right)=\lim _{\mu \rightarrow 0^{+}} \sigma(a, \mu) \geq \sigma(a, \overline{\mu})>0,
\end{align*}
which implies that $\hat{u} \not \equiv 0$, and in turn yields $\hat{\lambda}<0$.
Test (\ref{eq6.2}) and (\ref{eq6.3}) with $\hat{u}_{\mu}-\hat{u}$, and subtract, we have
\begin{align*} ||\Delta\left(\hat{u}_{\mu}\!-\!\hat{u}\right)||_2^{2}\!-\!\mu
\int_{\mathbb{R}^{N}}
\nabla \hat{u}_{\mu} \nabla\left(\hat{u}_{\mu}\!-\!\hat{u}\right)
\!-\!\int_{\mathbb{R}^{N}}\left(\hat{\lambda}_{\mu} \hat{u}_{\mu}\!-\!\hat{\lambda} \hat{u}\right) \left(\hat{u}_{\mu}\!-\!\hat{u}\right)=\!o(1).
\end{align*}
It result to $ ||\Delta\left(\hat{u}_{\mu}-\hat{u}\right)||_2^{2}-\hat{\lambda}
||\left(\hat{u}_{\mu}-\hat{u}\right)||_2^{2}=o(1)$, that is $\hat{u}_{\mu} \rightarrow \hat{u}$ in $H^{2}$. Moreover, we have $\sigma(a, 0)\!\leq\!E_{0}(\hat{u})=\mathop {\lim }\limits_{\mu \rightarrow 0^{+}}\sigma(a, \mu)\!\leq\!\sigma(a, 0)$. Consequently, $E_{0}(\hat{u})\!=\!\mathop {\lim }\limits_{\mu  \to 0^{+}}\sigma(a, \mu)\!=\!\sigma(a, 0)$ and $\hat{u}$ is a radial ground state to (\ref{eq6.3}).\qed \\

\noindent \textbf{Proof of Theorem \ref{th1.3}:}

The proof was motivated by Theorem 1.3 in \cite{Nbal} which dealing with $p\in(2,\bar{p})$, we give out the details since we consider $p\in(\bar{p},4^*)$. Let $a_k \to 0^+$ as $k\to+\infty$ and $u_k\in A_{R_0}^r$ be a minimizer of $m_r(a_k,\mu)$ for each $k\in \mathbb{N}$, where $A_{R_0}^r :=\left\{u \in S_{a_k}\cap H_{{rad}}^{2} : {||\Delta u||}_2<R_0(a_k,\mu) \right\}$. To begin with, we show that $\{\frac{\left\|\Delta u_{k}\right\|_{2}}{\left\|u_{k}\right\|_{2}}\}$ is bounded and $\frac{\left\|u_{k}\right\|_{p}^{p}}{\left\|u_{k}\right\|_{2}^{2}}\to0$ as $k\to+\infty$. Indeed, we derive from the Pohozaev identity $P_{\mu}(u_k)=0$ that
$$\left(\frac{1}{2}-\frac{1}{p\gamma_{p} }\right)||\Delta u_{k}||_{2}^{2}-\frac{\mu}{2}
\left(1-\frac{1}{p\gamma_{p}}\right)||\nabla u_{k}||_{2}^{2}=E_{\mu}\left(u_{k}\right)<-\frac{{a_k^2\mu}^2}{8}<0,$$
and hence $\left(\frac{1}{2}-\frac{1}{p\gamma_{p} }\right)\Big(\frac{||\Delta u_{k}||_{2}}{||u_{k}||_{2}}\Big)^2\leq \frac{\mu}{2}
\left(1-\frac{1}{p\gamma_{p}}\right) \frac{||\Delta u_{k}||_{2}}{||u_{k}||_{2}}$ implies $\frac{||\Delta u_{k}||_{2}}{||u_{k}||_{2}}\leq\frac{(p\gamma_{p}-1)\mu}{p\gamma_{p}-2}$. By the Gagliardo-Nirenberg inequality \eqref{2.9}, we have
$$\frac{\left\|u_{k}\right\|_{p}^{p}}{\left\|u_{k}\right\|_{2}^{2}} \!\leq \! C_{N,p}^p \Big(\frac{||\Delta u_{k}||_{2}}{||u_{k}||_{2}}\Big)^{p\gamma_p} {||u_k||}_2^{p\!-\!2} \!\leq \! C_{N,p}^p \Big[ \frac{(p\gamma_{p}\!-\!1)\mu}{p\gamma_{p}\!-\!2} \Big]^{p\gamma_p} a_k^{p\!-\!2}\to 0, ~~~~~~~~\text{as}~~~~ k \rightarrow +\infty. $$
From Lemma 3.1 in \cite{Nbal}, we know that $\inf_{u \in S_{a_k}}
\Big\{ \frac{1}{2}{||\Delta u||}_2^2\!-\!\frac{\mu}{2}{||\nabla u||}_2^2 \Big\}\!=-\frac{a_k^2\mu^2}{8}$. We also have $m_r(a_k,\mu)<-\frac{{a_k^2\mu}^2}{8}$ by Lemma \ref{lema3.4.1}. These facts imply that
\begin{align*}
-\frac{{a_k^2\mu}^2}{8}\!&>\!m_r(a_k,\mu)=\!E_{\mu}(u_k)\!= \frac{1}{2}{||\Delta u_k||}_2^2\!-\!\frac{\mu}{2}{||\nabla u_k||}_2^2 \!-\!\frac{1}{p}\left\|u_{k}\right\|_{p}^{p}\\
&\geq \! \inf _{u \in A_{R_{0}^r}} \Big\{ \frac{1}{2}{||\Delta u||}_2^2\!-\!\frac{\mu}{2}{||\nabla u||}_2^2 \Big\}\!-\!\frac{1}{p}\frac{\left\|u_{k}\right\|_{p}^{p}}
{\left\|u_{k}\right\|_{2}^{2}}a_k^2 \geq -\frac{{a_k^2\mu}^2}{8}-\!\frac{1}{p}\frac{\left\|u_{k}\right\|_{p}^{p}}
{\left\|u_{k}\right\|_{2}^{2}}a_k^2,
\end{align*}
whence Theorem \ref{th1.3}-$\textbf{(1)}$ follows.

Next, we prove Theorem \ref{th1.3}-$\textbf{(2)},\textbf{(3)}$. On the one hand, Theorem \ref{th1.1}-(3) implies that ${\tilde{\lambda}}_k<-\frac{{\mu}^2}{4}$. On the other hand, since $u_k$ solves \eqref{eq1.1}-\eqref{eq1.2}, we have
\begin{align*}
{\tilde{\lambda}}_k a_k^{2}\!=\!||\Delta u_{k}||_{2}^{2}\!-\!\mu ||\nabla u_{k}||_{2}^{2}\!-\!||u_{k}||_{p}^{p} \!\geq\! \inf _{u \in A_{R_{0}^r}} \Big\{ {||\Delta u||}_2^2\!-\!\mu{||\nabla u||}_2^2 \Big\}\!-\!||u_{k}||_{p}^{p}\!\geq\!-\frac{{a_k^2\mu}^2}{4}\!
-\!\frac{\left\|u_{k}\right\|_{p}^{p}}
{\left\|u_{k}\right\|_{2}^{2}}a_k^2,
\end{align*}
and so we get Theorem \ref{th1.3}-$\textbf{(2)}$. From Theorem \ref{th1.3}-$\textbf{(1)}$, we know that $m_r(a_k,\mu)\!=\!E_{\mu}(u_k)\!=\!-\frac{{a_k^2\mu}^2}{8}\!+\!o_k(1)$. Therefore, we have
\begin{align} \label{aln5.8}
-\frac{{\mu}^2}{8}\!+\!o_k(1)\!=\!
\frac{E_{\mu}(u_k)}{\left\|u_{k}\right\|_{2}^{2}}
\!=\!
\frac{1}{2}\frac{{||\Delta u_k||}_2^2}{\left\|u_{k}\right\|_{2}^{2}}\!-\!\frac{\mu}{2}\frac{{||\nabla u_k||}_2^2}{\left\|u_{k}\right\|_{2}^{2}}\!+\!o_k(1).
\end{align}
Recall that $P_{\mu}(u_k)=0$, it results to
\begin{align} \label{aln5.9}
0\!=\!\frac{P_{\mu}(u_k)}{\left\|u_{k}\right\|_{2}^{2}}\!=\!2\frac{{||\Delta u_k||}_2^2}{\left\|u_{k}\right\|_{2}^{2}}\!-\!\mu\frac{{||\nabla u_k||}_2^2}{\left\|u_{k}\right\|_{2}^{2}}\!-\!2{\gamma_p}
\frac{{||u_k||}_p^p}{\left\|u_{k}\right\|_{2}^{2}}
~~~~~~~~\Longrightarrow 2\frac{{||\Delta u_k||}_2^2}{\left\|u_{k}\right\|_{2}^{2}}\!-\!\mu\frac{{||\nabla u_k||}_2^2}{\left\|u_{k}\right\|_{2}^{2}}\!\to\!0.
\end{align}
Consequently, Theorem \ref{th1.3}-$\textbf{(3)}$ follows from \eqref{aln5.8}-\eqref{aln5.9}.

It remains to prove Theorem \ref{th1.3}-$\textbf{(4)}$. Let $v_k=\frac{u_k}{||u_k||_{2}}$, since $u_k$ solves \eqref{eq1.1}-\eqref{eq1.2}, we have
\begin{align*}
{\tilde{\lambda}}_k ||v_{k}||_{2}^{2}  \!=\!||\Delta v_{k}||_{2}^{2}\!-\!\mu ||\nabla v_{k}||_{2}^{2}\!-\!\frac{||u_{k}||_{p}^{p}}{||u_{k}||_{2}^{2}}\!=\!||\Delta v_{k}||_{2}^{2}\!-\!\mu ||\nabla v_{k}||_{2}^{2}+o_k(1).
\end{align*}
Let $\mathcal{F}v_k$ be the Fourier transform of $v_k$. By the fact that ${\tilde{\lambda}}_k\to-\frac{{\mu}^2}{4}$ and the Plancherel's formula, we know that
\begin{align*}
o_k(1)=||\Delta v_{k}||_{2}^{2}\!-\!\mu ||\nabla v_{k}||_{2}^{2}\!+\!\frac{{\mu}^2}{4} ||v_{k}||_{2}^{2}=\frac{1}{(2 \pi)^{N}}\int_{\mathbb{R}^{N}}\left(|\xi|^{2}-\frac{\mu}{2 }\right)^{2}\left|(\mathcal{F}v_k)(\xi)\right|^{2} d \xi,
\end{align*}
it results that $\int_{\mathbb{R}^{N}}\left(|\xi|^{2}-\frac{\mu}{2 }\right)^{2}\left|(\mathcal{F}v_k)(\xi)\right|^{2} d \xi \rightarrow 0$ and $\frac{1}{2}||\Delta v_{k}||_{2}^{2}\!-\!\frac{\mu}{2}||\nabla v_{k}||_{2}^{2} \to -\frac{{\mu}^2}{8}$ as $k \rightarrow +\infty$. Once again, by Lemma 3.1 in \cite{Nbal}, we have
$v_k\to 0$ in $L^q({\R}^N)$ for any $q\in(2,4^*)$. \qed \\

\vspace{.55cm}
\noindent {\bf Acknowledgements:} The authors would like to thank Professor Louis Jeanjean whose comments on the first version of this work have permitted us to improve our manuscript and to avoid including a wrong proof.

\end{document}